\documentclass[11pt,a4paper,reqno]{article}

\usepackage[utf8]{inputenc}
\usepackage[affil-it]{authblk}
\usepackage{amsfonts, amsmath, amssymb, amsthm, bbm, color, enumerate, graphicx, mathtools, tikz, hyperref, relsize}
\usepackage[numbers,sort]{natbib}
\usepackage[sc]{mathpazo}
\usepackage[T1]{fontenc}
\usepackage{eulervm}
\usepackage[a4paper,includeheadfoot, total={6in,10in}]{geometry}

\usepackage{subcaption} 

\usepackage{tikz-3dplot}
\usetikzlibrary{decorations.text, decorations.markings, mindmap,trees,shadows, decorations.pathreplacing}
\tikzstyle{doublearr}=[latex-latex,red, line width=0.5pt]
\tikzstyle{doublearr2}=[latex-latex,green!80!black, line width=0.5pt]
\usepackage{xcolor}

\newcommand{\qq}{\mathbf{q}}
\newcommand{\QQ}{\mathbf{Q}}
\newcommand{\bQ}{\bar{Q}}
\newcommand{\tQ}{\tilde{Q}}

\newcommand{\ZZ}{\mathbf{Z}}
\newcommand{\bZ}{\bar{\mathbbm{Z}}}	
\newcommand{\yy}{\mathbf{y}}			

\newcommand{\FF}{\mathbf{F}}

\newcommand{\xx}{\mathbf{x}}
\newcommand{\XX}{\mathbf{X}}
\newcommand{\YY}{\mathbf{Y}}

\def\sss{}
\newcommand{\Op}{\mathrm{O}_{\sss P}}
\newcommand{\op}{\mathrm{o}_{\sss P}}

\newcommand{\N}{\mathbbm{N}}
\newcommand{\R}{\mathbbm{R}}
\newcommand{\Z}{\mathbbm{Z}}
\newcommand\given[1][]{\:#1\middle|\:} 
\newcommand{\dif}{\ensuremath{\mbox{d}}}  
\newcommand{\e}{\mathrm{e}}

\newcommand{\pto}{\ensuremath{\xrightarrow{\mathbbm{P}}}}  
\newcommand{\dto}{\ensuremath{\xrightarrow{\mathcal{L}}}}  
\newcommand\norm[1]{\left\|#1\right\|_1}  

\newcommand\E[1]{\mathbbm{E}\left(#1\right)}  
\newcommand\ind[1]{\ensuremath{\mathbbm{1}_{\left[#1\right]}}} 
\newcommand\Pro[1]{\mathbbm{P}\left(#1\right)}  

\newtheorem{theorem}{Theorem}
\newtheorem{corollary}[theorem]{Corollary}
\newtheorem{lemma}[theorem]{Lemma}
\newtheorem{proposition}[theorem]{Proposition}
\newtheorem{remark}[theorem]{Remark}

\hypersetup{
 colorlinks,
 linkcolor=red,          
 citecolor=blue,       
 filecolor=red,   
 urlcolor=black, 
 pdftitle={},
 pdfauthor={},
 pdfcreator={Debankur Mukherjee},
 pdfsubject={},
 pdfkeywords={}
}

\usepackage{fancyhdr}
 
\tikzstyle{mybox} = [draw=red, fill=yellow!20, thick, minimum height=.4cm,
    rectangle, rounded corners]
\tikzstyle{fancytitle} =[fill=blue, text=white]
\usetikzlibrary{mindmap,trees,shadows}

\numberwithin{equation}{section}
\numberwithin{theorem}{section}

\newcommand{\MJSQ}{ \mathrm{ \sss MJSQ}}
\newcommand{\CJSQ}{ \mathrm{\sss CJSQ}}
\newcommand{\JSQ}{ \mathrm{\sss JSQ}}
\newcommand{\jsq}{\textnormal{JSQ}}

\tikzstyle{mybox} = [draw=black, thick, minimum height=.6cm,
    rectangle,text centered]
\tikzstyle{fancytitle} =[fill=blue, text=white]

\begin{document}

\title{Universality of Power-of-$d$ Load Balancing \\
in Many-Server Systems}

\author[1]{Debankur Mukherjee\footnote{\texttt{d.mukherjee@tue.nl}}}
\author[1,2]{Sem C.~Borst}
\author[1]{\\ Johan S.H.~van Leeuwaarden}
\author[3]{Philip A.~Whiting}
\affil[1]{
Eindhoven University of Technology, The Netherlands}
\affil[2]{Nokia Bell Labs, Murray Hill, NJ, USA}
\affil[3]{
Macquarie University, North Ryde, NSW, Australia}

\renewcommand\Authands{, }

\date{\today}

\maketitle 

\begin{abstract}
{\footnotesize
We consider a system of $N$~parallel single-server queues with unit exponential service rates and a single dispatcher where tasks arrive as a Poisson process of rate $\lambda(N)$. When a task arrives, the dispatcher assigns it to a server with the shortest queue among $d(N)$ randomly selected servers ($1 \leq d(N) \leq N$). This load balancing strategy is referred to as a JSQ($d(N)$) scheme, marking that it subsumes the celebrated Join-the-Shortest Queue (JSQ) policy as a crucial special case for $d(N) = N$.

We construct a stochastic coupling to bound the difference in the queue length processes between the JSQ policy and a JSQ($d(N)$) scheme with an arbitrary value of~$d(N)$. We use the coupling to derive the fluid limit in the regime where $\lambda(N) / N \to \lambda < 1$ as $N \to \infty$ with $d(N) \to\infty$, along with the associated fixed point. The fluid limit turns out not to depend on the exact growth rate of $d(N)$, and in particular coincides with that for the JSQ policy. We further leverage the coupling to establish that the diffusion limit in the critical regime where $(N - \lambda(N)) / \sqrt{N} \to \beta > 0$ as $N \to \infty$ with $d(N)/(\sqrt{N} \log (N))\to\infty$ corresponds to that for the JSQ policy. These results indicate that the optimality of the JSQ policy can be preserved at the fluid-level and diffusion-level while reducing the overhead by nearly a factor~O($N$) and O($\sqrt{N}/\log(N)$), respectively.
}
\end{abstract}

\section{Introduction}

In this paper we establish a universality property for a broad
class of randomized load balancing schemes in many-server systems.
While the specific features of load balancing policies may considerably
differ, the principal purpose is to distribute service requests or tasks
among servers or distributed resources in parallel-processing systems.
Well-designed load balancing schemes provide an effective mechanism
for improving relevant performance metrics experienced by users
while achieving high resource utilization levels.
The analysis and design of load balancing schemes has attracted
strong renewed interest in the last several years, mainly motivated
by significant challenges involved in assigning tasks
(e.g.~file transfers, compute jobs, database look-ups)
to servers in large-scale data centers.


In the present paper we focus on a basic scenario of sending tasks from a single dispatcher to $N$~parallel
queues with identical servers, exponentially distributed service
requirements, and a service discipline at each individual server that
is oblivious to the actual service requirements (e.g.~FCFS).
In this canonical case, the so-called Join-the-Shortest-Queue (JSQ)
policy has several strong optimality properties, and in particular
minimizes the overall mean delay among the class of non-anticipating
load balancing policies that do not have any advance knowledge of the
service requirements \cite{EVW80,W78,Winston77}.
(Relaxing any of the three above-mentioned assumptions tends to break
the optimality properties of the JSQ policy, and renders the
delay-minimizing policy quite complex or even counter-intuitive,
see for instance \cite{GHSW07,Jonckheere06,Whitt86}.)

In order to implement the JSQ policy, a dispatcher requires
instantaneous knowledge of the queue lengths at all the servers,
which may give rise to a substantial communication burden,
and not be scalable in scenarios with large numbers of servers.
The latter issue has motivated consideration of so-called JSQ($d$)
strategies, where the dispatcher assigns an incoming task to a server
with the shortest queue among $d$~servers selected uniformly at random.
Mean-field limit theorems in Mitzenmacher~\cite{Mitzenmacher01}
and Vvedenskaya {\em et al.}~\cite{VDK96} indicate that even a value as
small as $d = 2$ yields significant performance improvements
in a many-server regime with $N \to \infty$, in the sense that the tail
of the queue length distribution at each individual server falls off much
more rapidly compared to a strictly random assignment policy ($d = 1$).
This is commonly referred to as the ``power-of-two'' effect.
While these results were originally proved for exponential service
requirement distributions, they have been extended to general
service requirement distributions in Bramson {\em et al.}~\cite{BLP12}.
Analyses of several variants of this model can be found in~\cite{LM06, LN05, G05, BL12, EG16}

The diversity parameter~$d$ thus induces
a fundamental trade-off between the amount of communication overhead
and the performance in terms of queue lengths and delays.
Specifically, a strictly random assignment policy can be implemented
with zero overhead, but for any positive load per server,
the probability of non-zero wait and the mean waiting time do \emph{not}
fall to zero as $N \to \infty$.
In contrast, a nominal implementation of the JSQ policy (without
maintaining state information at the dispatcher) involves O($N$)
overhead per task, but it can be shown that the probability of non-zero
wait and the mean waiting time vanish as $N \to \infty$
for any fixed subcritical load per server.
Although JSQ($d$) strategies with a fixed parameter $d \geq 2$
yield significant performance improvements over purely random task
assignment while reducing the communication overhead by a factor O($N$)
compared to the JSQ policy, the probability of non-zero wait and mean waiting time do
\emph{not} vanish in the limit. 
In that sense a fixed value
of~$d$ is not sufficient to achieve asymptotically optimal performance.
This is also reflected by recent results of Gamarnik {\em et al.}~\cite{GTZ16} indicating that in the absence of any memory at the
dispatcher the communication overhead per task must grow with~$N$
in order to allow a zero mean waiting time in the limit.

In order to gain further insight in the trade-off between performance
and communication overhead as governed by the diversity parameter~$d$,
we also consider a regime where the number of servers~$N$ grows large,
but allow the value of~$d$ to depend on~$N$,
and write $d(N)$ to explicitly reflect that.
For convenience, we assume a Poisson arrival process of rate $\lambda(N)$
and unit-mean exponential service requirements.

We construct a stochastic coupling to bound the difference
in the queue length processes between the ordinary JSQ policy and a scheme
with an arbitrary value of $d(N)$.
We exploit the coupling to obtain the fluid limit 
in the subcritical regime where $\lambda(N) / N \to \lambda < 1$
as $N \to \infty$ with $d(N) \to\infty$, along with the associated fixed point.
As it turns out, the fluid limit does not depend on the exact growth rate
of $d(N)$, and in particular coincides with that for the JSQ policy. 
This implies that the overhead of the JSQ policy can be reduced by
`almost' a factor O($N$) while maintaining fluid-level optimality.
In case of batch arrivals fluid-level optimality can even be achieved with
$O(1)$ communication overhead per task.

We further consider the Halfin-Whitt heavy-traffic regime where
$(N - \lambda(N)) / \sqrt{N} \to \beta > 0$ as $N \to \infty$.
Recent work of Eschenfeldt \& Gamarnik~\cite{EG15} showed that the
diffusion-scaled system occupancy state for the ordinary JSQ policy
in this regime weakly converges to a two-dimensional reflected
Ornstein-Uhlenbeck process.
We leverage the above-mentioned coupling to prove that the diffusion limit
in case $d(N)/(\sqrt{N} \log(N))\to\infty$ as $N\to\infty$ corresponds to that for the JSQ policy.
This indicates that the overhead of the JSQ policy can
`almost' be reduced to O($\sqrt{N}\log N$) while retaining diffusion-level optimality. The above condition is in fact close to necessary, in the sense that
the diffusion-level behavior of the scheme is sub-optimal
if $d(N)/ (\sqrt{N} \log N)\to 0$ as $N\to\infty$.

The above results mirror the fluid-level and diffusion-level optimality properties reported in the companion paper~\cite{MBLW16-4}
for power-of-d($N$) strategies in a scenario with~$N$ server pools, where each server pool is a collection of servers, each working at unit rate.
The coupling developed in~\cite{MBLW16-4} has greater hold on the task completions, and provides absolute bounds on the difference of each component of the occupancy states.
More specifically, the task completions depend on the total number of active tasks in the entire system, whereas in the single-server scenario, it depends only on the number of non-idle servers.
As a result, obtaining a stochastic coupling bound in this paper becomes analytically more challenging, and in contrast with the infinite-server scenario, involves the cumulative loss terms and tail sums of the occupancy states of the ordinary JSQ policy.
This imposes the additional challenge of proving the $\ell_1$ convergence of the occupancy state process of the ordinary JSQ policy as will be described in greater detail later.
To the best of our knowledge, this is the first time the transient fluid limit of the ordinary JSQ policy is rigorously established.

The idea of using coupling to prove scaling limits of large-scale parallel-server systems was 
introduced by the authors in~\cite{MBLW15}.
The coupling method there was much weaker and was useful only for systems starting from specific initial occupancy states and for the particular scaling regime considered in that paper.
In contrast, in the current paper we need to develop a much stronger and wider coupling framework involving an intermediate class of schemes as described in Section~\ref{subsec:strategy} to establish the universality results.
In addition, we consider arbitrary starting states and different scaling regimes.
Remark~\ref{rem:jap-comp} further discusses the novelty and importance of the current stochastic comparison framework.

The remainder of the paper is organized as follows.
In Section~\ref{sec: model descr} we present a detailed model
description and state the main results, and in Section~\ref{sec:coupling} we construct a coupling and establish the stochastic ordering relations. Sections~\ref{sec:fluid} and~\ref{sec:diffusion} contain  the proofs of the fluid and diffusion limit results, respectively. 
Finally in Section~\ref{sec:conclusion} we
make some concluding remarks and briefly comment on future research directions.

\section{Main Results}\label{sec: model descr}

\subsection{Model description and notation}
Consider a system with $N$~parallel single-server queues with identical servers
and a single dispatcher.
Tasks with unit-mean exponential service requirements arrive at the
dispatcher as a Poisson process of rate $\lambda(N)$,
and are instantaneously forwarded to one of the servers.
Specifically, when a task arrives, the dispatcher assigns it
to a server with the shortest queue among $d(N)$ randomly selected
servers ($1 \leq d(N) \leq N$).
This load balancing strategy will
be referred to as the JSQ($d(N)$) scheme, marking that it subsumes
the ordinary JSQ policy as a crucial special case for $d(N) = N$.
The buffer capacity at each of the servers is~$b$ (possibly infinite),
and when a task is assigned to a server with $b$~pending tasks,
it is permanently discarded.

For any $d(N)$ ($1 \leq d(N) \leq N$), let 
$$\QQ^{\sss d(N)}(t): =
\left(Q_1^{\sss d(N)}(t), Q_2^{\sss d(N)}(t), \dots, Q_b^{\sss d(N)}(t)\right)$$ 
denote the
system occupancy state, where $Q_i^{\sss d(N)}(t)$ is the number of servers
under the JSQ($d(N)$) scheme with a queue length of~$i$ or larger,
at time~$t$, including the possible task in service, $i = 1, \dots, b$.
Figure~\ref{fig:1} provides a schematic diagram of the $Q_i$-values.
Throughout we assume that at each arrival epoch the servers are ordered
in  nondecreasing order of their queue lengths (ties can be broken arbitrarily), and whenever we refer
to some ordered server, it should be understood with respect to this prior ordering.

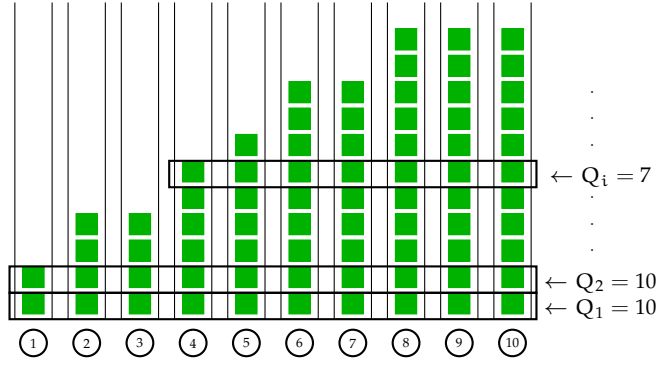
\begin{figure}
\begin{center}
\begin{tikzpicture}[scale=.70]
\foreach \x in {10, 9,...,1}
	\draw (\x,6)--(\x,0)--(\x+.7,0)--(\x+.7,6);
\foreach \x in {10, 9,...,1}
	\draw (\x+.35,-.45) node[circle,inner sep=0pt, minimum size=10pt,draw,thick] {{{\tiny $\mathsmaller{\x}$}}} ;
\foreach \y in {0, .5}
	\draw[fill=green!70!black,green!70!black] (1.15,.1+\y) rectangle (1.55,.5+\y);
\foreach \y in {0, .5, 1, 1.5}
	\draw[fill=green!70!black,green!70!black] (2.15,.1+\y) rectangle (2.55,.5+\y);
\foreach \y in {0, .5, 1, 1.5}
	\draw[fill=green!70!black,green!70!black] (3.15,.1+\y) rectangle (3.55,.5+\y);
\foreach \y in {0, .5, 1, 1.5, 2, 2.5}
	\draw[fill=green!70!black,green!70!black] (4.15,.1+\y) rectangle (4.55,.5+\y);
\foreach \y in {0, .5, 1, 1.5, 2, 2.5, 3}
	\draw[fill=green!70!black,green!70!black] (5.15,.1+\y) rectangle (5.55,.5+\y);
\foreach \y in {0, .5, 1, 1.5, 2, 2.5, 3, 3.5, 4}
	\draw[fill=green!70!black,green!70!black] (6.15,.1+\y) rectangle (6.55,.5+\y);
\foreach \y in {0, .5, 1, 1.5, 2, 2.5, 3, 3.5, 4}
	\draw[fill=green!70!black,green!70!black] (7.15,.1+\y) rectangle (7.55,.5+\y);
\foreach \y in {0, .5, 1, 1.5, 2, 2.5, 3, 3.5, 4, 4.5, 5}
	\draw[fill=green!70!black,green!70!black] (8.15,.1+\y) rectangle (8.55,.5+\y);
\foreach \y in {0, .5, 1, 1.5, 2, 2.5, 3, 3.5, 4, 4.5, 5}
	\draw[fill=green!70!black,green!70!black] (9.15,.1+\y) rectangle (9.55,.5+\y);
\foreach \y in {0, .5, 1, 1.5, 2, 2.5, 3, 3.5, 4, 4.5, 5}
	\draw[fill=green!70!black,green!70!black] (10.15,.1+\y) rectangle (10.55,.5+\y);

\draw[thick] (.9,0) rectangle (10.8,.5);
\draw[thick] (.9,.5) rectangle (10.8,1);
\draw[thick] (3.9,2.5) rectangle (10.8,3);

\draw  (12, .2) node {{\scriptsize $\leftarrow Q_1=10$}};
\draw  (12, .7) node {{\scriptsize $\leftarrow Q_2=10$}};

\draw  (11.85, 1.3) node {{\tiny $\cdot$}};
\draw  (11.85, 1.8) node {{\tiny $\cdot$}};
\draw  (11.85, 2.3) node {{\tiny $\cdot$}};
\draw  (12, 2.7) node {{\scriptsize $\leftarrow Q_i=7$}};
\draw  (11.85, 3.3) node {{\tiny $\cdot$}};
\draw  (11.85, 3.8) node {{\tiny $\cdot$}};
\draw  (11.85, 4.3) node {{\tiny $\cdot$}};

\end{tikzpicture}
\end{center}
\caption{The occupancy state of the system; When the servers are arranged in nondecreasing order of their queue lengths, $Q_i$ represents the width of the $i^\mathrm{\mathrm{th}}$ row.}
\label{fig:1}
\end{figure}

We occasionally omit the superscript $d(N)$, and replace it by~$N$, to refer
the $N^{\mathrm{th}}$ system, when the value of $d(N)$ is clear from the context.
When a task is discarded, in case of a finite buffer size,
we call it an \emph{overflow} event, and we denote by $L^{\sss d(N)}(t)$ the
total number of overflow events under the JSQ($d(N)$) scheme up to time~$t$.

A sequence of random variables $\big\{X_N\big\}_{N\geq 1}$, for some function $f:\R\to\R_+$, is said to be $\Op(f(N))$, if the sequence of scaled random variables $\big\{X_N/f(N)\big\}_{N\geq 1}$ is  tight, or said to be $\op(f(N))$, if $\big\{X_N/f(N)\big\}_{N\geq 1}$ converges to zero in probability.
Boldfaced letters are used to denote vectors.
We denote by $\ell_1$ the space of all summable sequences. 
For any set $K$, the closure is denoted by $\overline{K}$. 
We denote by $D_E[0,\infty)$ the set of all \emph{c\'adl\'ag} (right continuous left limit exists) functions from $[0,\infty)$ to a complete separable metric space $E$, and
by `$\dto$'  convergence in distribution for real-valued random variables and with respect to the Skorohod-$J_1$ topology for c\'adl\'ag processes.

\subsection{Fluid-limit results}\label{sec:fluidresult}

In the fluid-level analysis, we consider the subcritical regime
where $\lambda(N) / N \to \lambda < 1$ as $N \to \infty$.
In order to state the results, we first introduce some useful notation.
Denote the fluid-scaled system occupancy state by
$\qq^{\sss d(N)}(t) := \QQ^{\sss d(N)}(t) / N$, i.e., $q^{\sss d(N)}_i(t)=Q^{\sss d(N)}_i(t)/N$, and define 
$$\mathcal{S} =
\left\{\qq \in [0, 1]^b: q_i \leq q_{i-1} \mbox{ for all } i = 2, \dots, b, \mbox{ and } \sum_{i=1}^b q_i < \infty\right\}$$
as the set of all possible fluid-scaled occupancy states equipped with $\ell_1$ topology.
For any $\qq \in \mathcal{S}$, denote $m(\qq) = \min\{i: q_{i + 1} < 1\}$,
with the convention that
$q_{b+1} = 0$ if $b < \infty$.
Note that $m(\qq)<\infty$, since $\qq\in\ell_1$.
If $m(\qq)=0$, then define $p_{0}(\qq)=1$ and $p_i(\qq) = 0$ for all $i \geq 1$.
If $m(\qq)>0$, distinguish two cases, depending on whether the normalized arrival
rate $\lambda$ is larger than $1 - q_{m(\qq) + 1}$ or not.
If $\lambda < 1 - q_{m(\qq) + 1}$, then define $p_{m(\qq) - 1}(\qq) = 1$
and $p_i(\qq) = 0$ for all $i \neq m(\qq) - 1$.
On the other hand, if $\lambda > 1 - q_{m(\qq) + 1}$,
then $p_{m(\qq) - 1}(\qq) = (1 - q_{m(\qq) + 1}) / \lambda$,
$p_{m(\qq)}(\qq) = 1 - p_{m(\qq) - 1}(\qq)$,
and $p_i(\qq) = 0$ for all $i \neq m(\qq) - 1, m(\qq)$.
Note that the assumption $\lambda < 1$ ensures that the latter case
cannot occur when $m(\qq) = b<\infty$.

\begin{theorem}{\normalfont (Universality of fluid limit for JSQ($d(N)$) scheme)}
\label{fluidjsqd}
Assume $\qq^{\sss d(N)}(0) \to \qq^\infty$ in~$\mathcal{S}$ and $\lambda(N)/N\to\lambda<1$ as $N \to \infty$.
For the \jsq$(d(N))$ scheme with $d(N) \to\infty$, any subsequence of the sequence of processes
$\big\{\qq^{\sss d(N)}(t)\big\}_{t \geq 0}$ has a further subsequence that converges weakly with respect to the Skorohod $J_1$ topology, to the limit $\big\{\qq(t)\big\}_{t \geq 0}$ 
satisfying the following system of integral equations
\begin{equation}\label{eq:fluid}
 q_i(t) = q_i^\infty+ \lambda \int_0^t p_{i-1}(\qq(s))\dif s - \int_0^t(q_i(s) - q_{i+1}(s))\dif s,\quad i=1,\ldots,b,
\end{equation}
where the coefficients $p_i(\cdot)$ are as
defined earlier.
\end{theorem}

The above theorem shows that the fluid-level dynamics do not depend
on the specific growth rate of $d(N)$ as long as $d(N) \to \infty$
as $N \to \infty$.
In particular, the JSQ($d(N)$) scheme with $d(N)\to\infty$ exhibits the
same behavior as the ordinary JSQ policy  in the limit, and thus achieves fluid-level
optimality.

The coefficient $p_i(\qq)$ represents the instantaneous fraction
of incoming tasks assigned to servers with a queue length of exactly~$i$
in the fluid-level state $\qq \in \mathcal{S}$.
Assuming $m(\qq) < b$, a strictly positive fraction $1 - q_{m(\qq) + 1}$
of the servers have a queue length of exactly $m(\qq)$.
Since $d(N) \to \infty$, the fraction of incoming tasks that get assigned
to servers with a queue length of $m(\qq) + 1$ or larger is zero:
$p_i(\qq) = 0$ for all $i = m(\qq) + 1, \dots, b - 1$.
Also, tasks at servers with a queue length of exactly~$i$ are completed
at (normalized) rate $q_i - q_{i + 1}$, which is zero for all
$i = 0, \dots, m(\qq) - 1$, and hence the fraction of incoming tasks
that get assigned to servers with a queue length of $m(\qq) - 2$ or less
is zero as well: $p_i(\qq) = 0$ for all $i = 0, \dots, m(\qq) - 2$.
This only leaves the fractions $p_{m(\qq) - 1}(\qq)$
and $p_{m(\qq)}(\qq)$ to be determined.
Now observe that the fraction of servers with a queue length of exactly
$m(\qq) - 1$ is zero.
If $m(\qq)=0$, then clearly the incoming tasks will join the empty queue, 
and thus, $p_{m(\qq)}=1$, and $p_i(\qq) = 0$ for all $i \neq m(\qq)$.
Furthermore, if $m(\qq)\geq 1$, since tasks at servers with a queue length of exactly $m(\qq)$
are completed at (normalized) rate $1 - q_{m(\qq) + 1} > 0$,
incoming tasks can be assigned to servers with a queue length of exactly
$m(\qq) - 1$ at that rate.
We thus need to distinguish between two cases, depending on whether the
normalized arrival rate $\lambda$ is larger than $1 - q_{m(\qq) + 1}$ or not.
If $\lambda < 1 - q_{m(\qq) + 1}$, then all the incoming tasks can be
assigned to a server with a queue length of exactly $m(\qq) - 1$,
so that $p_{m(\qq) - 1}(\qq) = 1$ and $p_{m(\qq)}(\qq) = 0$.
On the other hand, if $\lambda > 1 - q_{m(\qq) + 1}$, then not all
incoming tasks can be assigned to servers with a queue length of
exactly $m(\qq) - 1$ active tasks, and a positive fraction will be
assigned to servers with a queue length of exactly $m(\qq)$:
$p_{m(\qq) - 1}(\qq) = (1 - q_{m(\qq) + 1}) / \lambda$
and $p_{m(\qq)}(\qq) = 1 - p_{m(\qq) - 1}(\qq)$.

It is easily verified that the unique fixed point $\qq^\star = (q_1^\star,q_2^\star,\ldots, q_b^\star)$ of the system of differential
equations in~\eqref{eq:fluid} is given by
\begin{equation}
\label{eq:fixed point}
q_i^* = \left\{\begin{array}{ll} \lambda, & i = 1, \\
0, & i =  2, \dots, b. \end{array} \right.
\end{equation}
Note that the fixed point in~\eqref{eq:fixed point} is consistent with the results in \cite{Mitzenmacher01,VDK96,YSK15}
for fixed~$d$, where taking $d \to \infty$ yields the same fixed point.
However, the results in \cite{Mitzenmacher01,VDK96,YSK15} for fixed~$d$
cannot be directly used to handle joint scalings, and do not yield the
universality of the entire fluid-scaled sample path
for arbitrary initial states as established in Theorem~\ref{fluidjsqd}.

The fixed point in~\eqref{eq:fixed point} in conjunction with the interchange of limits result in Proposition~\ref{th:interchange} below indicates that in stationarity the fraction of servers with a queue
length of two or larger is negligible.
Let 
$$\pi^{\sss d(N)}(\cdot)=\lim_{t\to\infty}\Pro{\qq^{\sss d(N)}(t)=\cdot}$$ 
be the stationary measure of the occupancy states of the $N^{\mathrm{th}}$ system.  
\begin{proposition}{{\normalfont (Interchange of limits)}}
\label{th:interchange}
For the JSQ$(d(N))$ scheme 
let $\pi^{\sss d(N)}$ be the stationary measure of the occupancy states of the $N^{\mathrm{th}}$ system. 
Then $\pi^{\sss d(N)}\dto\pi^\star$ as $N\to\infty$ with $d(N)\to\infty$, where $\pi^\star=\delta_{\qq^\star}$ with $\delta_x$ being the Dirac measure concentrated upon $x$, and $\qq^\star$ as in~\eqref{eq:fixed point}.
\end{proposition}
The above proposition relies on tightness of $\big\{\pi^{d(N)}\big\}_{N\geq 1}$ and the global stability of the fixed point, and is proved in Subsection~\ref{ssec:globstab}.
\\

We now consider an extension of the model in which tasks arrive in batches. We assume that the batches arrive as a Poisson process with rate  $\lambda(N)/\ell(N)$, and have fixed size $\ell(N)>0$, so that the effective total task arrival rate remains $\lambda(N)$. 
We will show that even for arbitrarily slowly growing batch size,  fluid-level optimality can be achieved with $O(1)$ communication overhead per task. 
For that, we define the JSQ($d(N)$) scheme adapted for batch arrivals. When a batch of size $\ell(N)$ arrives, the dispatcher samples $d(N)\geq \ell(N)$ servers without replacement, and assigns the $\ell(N)$ tasks to the $\ell(N)$ servers with the smallest queue length among the sampled servers.

\begin{theorem}{\normalfont (Batch arrivals)}
\label{th:batch}
Consider the batch arrival scenario with growing batch size $\ell(N)\to\infty$ and $\lambda(N)/N\to\lambda<1$ as $N\to\infty$. For the \jsq$(d(N))$ scheme with $d(N)\geq \ell(N)/(1-\lambda-\varepsilon)$ for any fixed $\varepsilon>0$, if $q^{\sss d(N)}_1(0)\pto q_1^\infty\leq \lambda$, and $q_i^{\sss d(N)}(0)\pto 0$ for all $i\geq 2$, then the sequence of processes
$\big\{\qq^{\sss d(N)}(t)\big\}_{t\geq 0}$ converges weakly to the limit $\big\{\qq(t)\big\}_{t\geq 0}$, described as follows: 
\begin{equation}\label{eq:batch}
q_1(t) = \lambda + (q_1^\infty-\lambda)\e^{-t},\quad
q_i(t)\equiv 0\quad \mathrm{for\ all}\quad i= 2,\ldots,b.
\end{equation}
\end{theorem}
The fluid limit in~\eqref{eq:batch} agrees with the fluid limit of the JSQ$(d(N))$ scheme if the initial state is taken as in Theorem~\ref{th:batch}. 
Further observe that the fixed point also coincides with that of the JSQ policy, as given by \eqref{eq:fixed point}.
Also, for a fixed $\varepsilon>0$, the communication overhead per task is on average given by $(1-\lambda-\varepsilon)^{-1}$ which is $O(1)$.
Thus Theorem~\ref{th:batch} ensures that in case of batch arrivals with growing batch size, fluid-level optimality can be achieved with $O(1)$ communication overhead per task.
The result for the fluid-level optimality in stationarity can also be obtained indirectly by exploiting the fluid-limit result in~\cite{YSK15}.
Specifically, it can be deduced from the result in~\cite{YSK15} that for batch arrivals with growing batch size, the JSQ$(d(N))$ scheme with suitably growing $d(N)$ yields the same fixed point of the fluid limit as described in~\eqref{eq:fixed point}.

\subsection{Diffusion-limit results}

In the diffusion-limit analysis, we consider the Halfin-Whitt
regime where 
$$\frac{N - \lambda(N)}{ \sqrt{N}} \to \beta\quad\text{as}\quad N \to \infty$$
for some positive coefficient $\beta > 0$.
In order to state the results, we first introduce some useful notation.
Let $\bar{\QQ}^{\sss d(N)}(t) =
\big(\bar{Q}_1^{\sss d(N)}(t), \bar{Q}_2^{\sss d(N)}(t), \dots, \bar{Q}_b^{\sss d(N)}(t)\big)$
be a properly centered and scaled version of the system occupancy state
$\QQ^{\sss d(N)}(t)$, with 
$$\bar{Q}_1^{\sss d(N)}(t) = - \frac{N-Q_1^{\sss d(N)}(t)}{ \sqrt{N}},\qquad \bar{Q}_i^{\sss d(N)}(t) =\frac{ Q_i^{\sss d(N)}(t)}{\sqrt{N}},\quad i = 2, \dots, b.$$
The reason why $Q_1^{\sss d(N)}(t)$ is centered around~$N$ while $Q_i^{\sss d(N)}(t)$,
$i = 2, \dots, b$, are not, is because the fraction of servers with a queue
length of exactly one tends to one, whereas the fraction of servers with a queue length of two or more tends to zero as $N\to\infty$.

\begin{theorem}{\normalfont (Universality of diffusion limit for JSQ($d(N)$) scheme)}
\label{diffusionjsqd}
Assume $\bQ_i^{\sss d(N)}(0) \dto \bQ_i(0)$ in $\mathbbm{R}$ as $N \to \infty$,
buffer capacity $b\geq 2$ (possibly infinite),
and there exists some $k\geq 2$ such that $\bQ_{k+1}^N(0)=0$ for all sufficiently large $N$.
For $d(N) /( \sqrt{N} \log N)\to\infty$,
the sequence of processes $\big\{\bar{\QQ}^{\sss d(N)}(t)\big\}_{t \geq 0}$ 
converges weakly to the limit
 $\big\{\bar{\QQ}(t)\big\}_{t \geq 0}$ in~$D_{\mathcal{S}}[0,\infty)$,
where $\bar{Q}_i(t) \equiv 0$ for $i \geq k+1$
and $(\bar{Q}_1(t), \bar{Q}_2(t),\ldots,\bQ_k(t))$
are the unique solutions in $D_{\R^k}[0,\infty)$ of the stochastic integral equations
\begin{equation}\label{eq:diffusionjsqd}
\begin{split}
\bar{Q}_1(t) &= \bar{Q}_1(0) + \sqrt{2} W(t) - \beta t +
\int_0^t (- \bar{Q}_1(s) + \bar{Q}_2(s)) \dif s - U_1(t), \\
\bar{Q}_2(t) &= \bar{Q}_2(0) + U_1(t) - \int_0^t (\bar{Q}_2(s)-\bar{Q}_3(s) )\dif s, \\
\bQ_i(t)&=\bQ_i(0)-\int_0^t (\bQ_i(s) - \bQ_{i+1}(s))\dif s,\quad i= 3,\ldots,k
\end{split}
\end{equation}
for $t \geq 0$, where $W$ is the standard Brownian motion and $U_1$ is
the unique nondecreasing nonnegative process in~$D_\R[0,\infty)$ satisfying
$\int_0^\infty \mathbbm{1}_{[\bar{Q}_1(t) < 0]} \dif U_1(t) = 0$.

\end{theorem}
Although~\eqref{eq:diffusionjsqd} differs from the diffusion limit obtained for the fully pooled M/M/N model in the Halfin-Whitt regime~\cite{HW81, LK11, LK12}, it shares similar favorable properties.
Observe that $-\bQ_1^{\sss d(N)}$ is the scaled number of vacant servers. 
Thus, Theorem~\ref{diffusionjsqd} shows that over any finite time horizon,
there will be $O_P(\sqrt{N})$ servers with queue length zero
and $O_P(\sqrt{N})$ servers with a queue length larger than two,
and hence all but $O_P(\sqrt{N})$ servers have a queue length of exactly one.
This diffusion limit is proved in~\cite{EG15} for the ordinary JSQ policy, and its steady-state properties are 
studied in~\cite{Braverman18, BM18, BM18b}.
Our contribution is to construct a stochastic coupling
and establish that, somewhat remarkably, the diffusion limit is the
same for any JSQ($d(N)$) scheme, as long as $d(N) /( \sqrt{N} \log(N))\to\infty$.
In particular, the JSQ($d(N)$) scheme with $d(N) /( \sqrt{N} \log(N))\to\infty$ exhibits the
same behavior as the ordinary JSQ policy in the limit, and thus achieves
diffusion-level optimality.
This growth condition for $d(N)$ is not only sufficient, but also nearly necessary,
as indicated by the next theorem.

\begin{theorem}{\normalfont (Almost necessary condition)}
\label{th:diff necessary}
Assume $\bQ_i^{\sss d(N)}(0) \dto \bQ_i(0)$ in $\mathbbm{R}$ as $N \to \infty$. If $d(N)/(\sqrt{N}\log N)\to 0$, then the diffusion limit of the \jsq$(d(N))$ scheme differs from that of the JSQ policy.
\end{theorem}

Theorem~\ref{th:diff necessary}, in conjunction with Theorem~\ref{diffusionjsqd}, shows that $\sqrt{N}\log N$ is the minimal order of $d(N)$ for the JSQ$(d(N))$ scheme to achieve diffusion-level optimality.

\subsection{Proof strategy}\label{subsec:strategy}
The idea behind the proofs of the asymptotic results for the JSQ$(d(N))$ scheme in Theorems~\ref{fluidjsqd} and~\ref{diffusionjsqd} is to 
(i)~prove the fluid limit and exploit the existing diffusion limit result for the ordinary JSQ policy, and then (ii)~prove a universality result by establishing that the ordinary JSQ policy and the JSQ$(d(N))$ scheme coincide under some suitable conditions on $d(N)$. 
For the ordinary JSQ policy the fluid limit in the subcritical regime is established in~Subsection~\ref{ssec:fluidjsq}, and the diffusion limit in the Halfin-Whitt heavy-traffic regime in~\cite[Theorem 2]{EG15}.
A direct comparison between the JSQ$(d(N))$ scheme and the ordinary JSQ policy is not straightforward, which is why we introduce the CJSQ$(n(N))$ class of schemes as an intermediate scenario to establish the universality result.

Just like the JSQ$(d(N))$ scheme, the schemes in the class CJSQ$(n(N))$ may be thought of as ``sloppy'' versions of the JSQ policy, in the sense that tasks are not necessarily assigned to a server with the shortest queue length but to one of the $n(N)+1$
lowest ordered servers, as graphically illustrated in Figure~\ref{fig:sfig1}.
In particular, for $n(N)=0$, the class only includes the ordinary JSQ policy. 
Note that the JSQ$(d(N))$ scheme is guaranteed to identify the lowest ordered server, but only among a randomly sampled subset of $d(N)$ servers.
In contrast, a scheme in the CJSQ$(n(N))$ class only guarantees that
one of the $n(N)+1$ lowest ordered servers is selected, but 
across the entire pool of $N$ servers. 
We will show that for sufficiently small $n(N)$, any scheme from the class CJSQ$(n(N))$ is still `close' to the ordinary JSQ policy. 
We will further prove that for sufficiently large $d(N)$ relative to $n(N)$ we can construct a scheme
called JSQ$(n(N),d(N))$, belonging to the CJSQ$(n(N))$ class, which differs `negligibly' from the JSQ$(d(N))$ scheme. 
Therefore,  for a `suitable' choice of $d(N)$ the idea is to produce a `suitable' $n(N)$.
This proof strategy is schematically represented in Figure~\ref{fig:sfig3}.
\begin{figure}
\begin{center}
\begin{subfigure}{.5\textwidth}
  \centering
  \begin{tikzpicture}[scale=.6]
\draw (1,0)--(1,2)--(3,2)--(3,2.5)--(5,2.5)--(5,3)--(7,3)--(7,3.5)--(8,3.5)--(8,4)--(10,4)--(10,4.5)--(11,4.5)--(11,0)--(1,0);

\draw  (1,.5)--(11,.5);
\draw  (1,1)--(11,1);
\draw  (1,1.5)--(11,1.5);
\draw  (1,2)--(11,2);
\draw  (3,2.5)--(11,2.5);
\draw  (5,3)--(11,3);
\draw  (7,3.5)--(11,3.5);
\draw  (10,4)--(11,4);

\draw[very thick] (1,6)--(1,0)--(11,0)--(11,6);

\draw[fill=green!70!black,opacity=0.3] (1.05,0) rectangle (2,6);
\draw[dashed] (1.5,5.5) -- (1.5, 2.25);
\draw[dashed, decoration={markings,mark=at position 1 with
    {\arrow[scale=1.2,>=stealth]{>}}},postaction={decorate}] (1.5, 2.25) -- (2.6, 2.25);

\draw [decorate,decoration={brace,amplitude=3pt},xshift=0pt,yshift=2pt]
(1,6.25) -- (2,6.25) node [black,midway,yshift=0.3cm] { $\scriptscriptstyle n(N)+1$};

\end{tikzpicture}
  \caption{CJSQ$(n(N))$ scheme\vspace{16pt}}
  \label{fig:sfig1}
\end{subfigure}%
\begin{subfigure}{.5\textwidth}
  \centering
  \begin{tikzpicture}[scale=.6]
  
  \draw(14,1) node[mybox, text width = 2.5cm,text centered]  {%
   {\scriptsize JSQ$(n(N),d(N))$}};
   
  \draw(22,1) node[mybox, text width = 2.5cm,text centered]  {%
   {\scriptsize  CJSQ$(n(N))$}};
   
   \draw(14,6) node[mybox, text width = 2.5cm,text centered]  {%
   {\scriptsize  JSQ$(d(N))$}};
   
   \draw(22,6) node[mybox, text width = 2.5cm,text centered]  {%
   {\scriptsize  JSQ}};

\draw[doublearr] (16.5,6) to (19.5,6);
\draw[doublearr2] (14,2) to (14,5);
\draw[doublearr2] (16.5,1) to (19.5,1);
\draw[doublearr2] (22,2) to (22,5);

\node [rotate=270] at (22.5,3.5) {\tiny Suitable $n(N)$};
\node [rotate=90] at (13.5,3.5) {\tiny Suitable $d(N)$};
\draw[text width=1.1cm, align=center] (18,1.75) node {\tiny Belongs to};
\draw[text width=1.1cm, align=center] (18,1.35) node {\tiny the class};
\end{tikzpicture}
  \caption{Asymptotic equivalence relations}
  \label{fig:sfig3}
\end{subfigure}
\caption{(a) High-level view of the CJSQ$(n(N))$ class of schemes, where as in Figure~\ref{fig:1}, the servers are arranged in the nondecreasing order of their queue lengths, and the arrival must be assigned through the left tunnel. (b) The equivalence structure is depicted for various intermediate load balancing schemes to facilitate the comparison between the JSQ$(d(N))$ scheme and the ordinary JSQ policy.
}
\label{fig:strategy}
\end{center}
\end{figure}
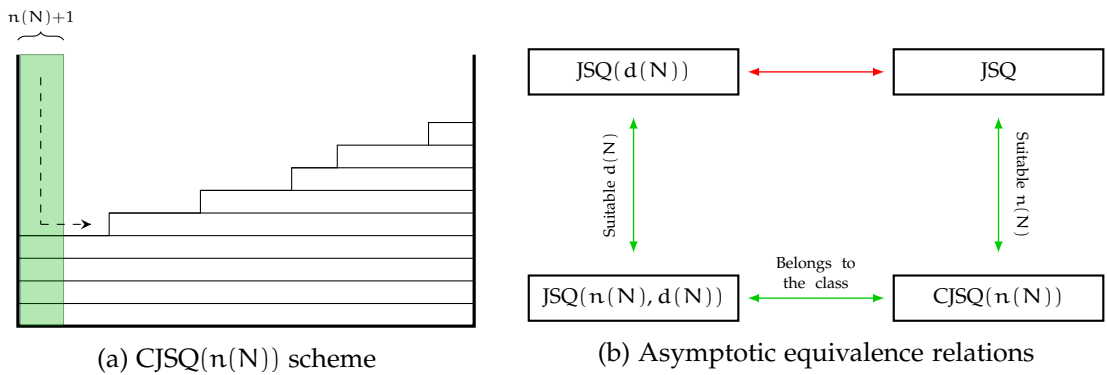

\begin{remark}\label{rem:jap-comp}\normalfont
As mentioned in the introduction, a coupling method was used in~\cite{MBLW15} to establish the diffusion limit of the Join-the-Idle Queue (JIQ) policy starting from specific initial occupancy states.
Comparing the JIQ and JSQ policies in that scaling regime was much easier when viewed as follows:
(i) If there is an idle server in the system, both JIQ and JSQ perform similarly,
(ii)~Also, when there is no idle server and only $O(\sqrt{N})$ servers with queue length two, JSQ assigns the arriving task to a server with queue length one. 
In that case, since JIQ assigns at random, the probability that the task will land on a server with queue length two and thus acts differently than JSQ is $O(1/\sqrt{N})$.
Since on any finite time interval the number of times an arrival finds all servers busy is at most $O(\sqrt{N})$, all the arrivals except an $O(1)$ of them are assigned in exactly the same manner in both JIQ and JSQ, which then leads to the same scaling limit for both policies.
Note that in the computation of the expected number of events when JIQ and JSQ performs differently, both the specific initial state condition and the scaling regime were crucial.
In the current paper the stochastic comparison framework is inherently different.
Here the idea pivots on two key observations: 
(i)~For any scheme, if each arrival is assigned to \emph{approximately} the shortest queue, then the scheme can still retain its optimality on various scales, and
(ii)~For any two schemes, if on any finite time interval not \emph{too many} arrivals are assigned to different ordered servers, then they can have the same scaling limits. 
Combination of the above two ideas provides a much wider coupling framework involving an intermediate class of schemes that enables us to consider arbitrary starting states and different scaling regimes.
In addition, the consideration of the arbitrary starting state will turn out to be crucial in order to extend the fluid-scale universality result to the steady state.
\end{remark}

In the next section we construct a stochastic coupling called S-coupling, which will be the key vehicle in establishing the universality result mentioned above.
\begin{remark} \normalfont
Observe that, sampling \emph{without} replacement polls more servers than \emph{with} replacement, and hence the minimum number of active tasks among the selected servers is stochastically smaller in the case without replacement.
As a result, for sufficient conditions as in Theorems~\ref{fluidjsqd} and~\ref{th:diff necessary}, it is enough to consider sampling with replacement.
Also, for notational convenience, in the proof of the almost necessary condition stated in Theorem~\ref{th:diff necessary} we will assume sampling with replacement, although the proof technique and the result is valid if the servers are chosen without replacement.
\end{remark} 

\section{Coupling and Stochastic Ordering}\label{sec:coupling}

In this section, we construct a coupling between any scheme from the class CJSQ($n(N)$) and the ordinary JSQ policy, which ensures that for sufficiently small $n(N)$, on any finite time interval, the two schemes differ negligibly. This plays an instrumental role in establishing  the universality results in Theorems~\ref{fluidjsqd} and~\ref{diffusionjsqd}.
All the statements in this section should be understood to apply to the $N^{\mathrm{th}}$ system with $N$ servers.

\subsection{Stack formation and deterministic ordering}
In order to prove the stochastic comparisons among the various schemes, as in~\cite{MBLW15}, we describe the many-server system as an ensemble of stacks, in a way that two different ensembles can be ordered.
In this formulation, at each step, items are added or removed according to some rule. From a high level,  we then show that if two systems follow some specific rules, then at any step, the two ensembles maintain some kind of deterministic ordering. 
This deterministic ordering turns into an almost sure ordering in the next subsection, when we construct the S-coupling.

Each server along with its queue is thought of as a stack of items, and we always consider the stacks to be arranged in nondecreasing order of their heights. 
The ensemble of stacks then represents the empirical CDF of the queue length distribution, and the $i^{\mathrm{th}}$ horizontal bar corresponds to $Q_i^{\Pi}$ (for some task assignment scheme $\Pi$), as depicted in Figure~\ref{fig:1}.
 If an arriving item happens to land on a stack which already contains $b$ items, then the item is discarded, and is added to a special stack $L^\Pi$ of discarded items, where it stays forever.

Any two ensembles $A$ and $B$, each having $N$ stacks and a maximum height $b$ per stack, are said to follow Rule($n_A,n_B,k$) at some step, 
if either an item is removed from the $k^{\mathrm{th}}$ stack in both ensembles (if nonempty), 
or an item is added to the $n_A^{\mathrm{th}}$ stack in ensemble $A$ and to the $n_B^{\mathrm{th}}$ stack in ensemble $B$. 

\begin{proposition}\label{prop:det ord}
For any two ensembles of stacks $A$ and $B$, as described above, if at any step \emph{Rule}$(n_A,n_B,k)$ is followed for some value of $n_A$, $n_B$, and $k$, with $n_A\leq n_B$, then the following ordering is always preserved: for all $m\leq b$,
\begin{equation}\label{eq:det ord}
\sum_{i=m}^b Q_i^A +L^A\leq \sum_{i=m}^b Q_i^B +L^B.
\end{equation}
\end{proposition} 

This proposition says that, while adding the items to the ordered stacks, if we ensure that in ensemble $A$ the item is always placed to the left of that in ensemble $B$, and if the items are removed from the same ordered stack in both ensembles, then the 
aggregate size of the $b-m+1$ highest horizontal bars as depicted in Figure~\ref{fig:1} plus the cumulative number of discarded items is no larger in $A$ than in $B$ throughout.

\begin{proof}[Proof of Proposition~\ref{prop:det ord}]
We prove the ordering by forward induction on the time-steps, i.e., we assume that at some step the ordering holds, and show that in the next step it will be preserved. 
In ensemble $\Pi$, where $\Pi=A$, $B$, after applying Rule($n_A,n_B,k$), the updated lengths of the horizontal bars are denoted by $\tQ^\Pi_i$, $i\geq 1$. Also, define $I_\Pi(c):=\max\big\{i\geq 0:Q_i^\Pi\geq N-c+1\big\}$, $c=1,\ldots,N$, with the convention that $Q_0^{\Pi}\equiv N$.

Now if the rule prescribes removal of an item from the $k^{\rm th}$ stack,  then the updated ensemble will have the values
\begin{equation}\label{eq:removal}
\tilde{Q}^{\Pi}_i=
\begin{cases}
Q^{\Pi}_i-1, &\mbox{ for }i=I_{\Pi}(k),\\
Q^{\Pi}_j,&\mbox{ otherwise, }
\end{cases}
\end{equation}
if $I_{\Pi}(k)\geq 1$; otherwise all the $Q^{\Pi}_i$-values remain unchanged. On the other hand, if the rule produces the addition of an item to stack $n_{\sss\Pi}$, then the values will be updated as
\begin{equation}\label{eq:addition}
\tilde{Q}^{\Pi}_i=
\begin{cases}
Q^{\Pi}_i+1, &\mbox{ for }i=I_{\Pi}(n_{\sss\Pi})+1,\\
Q^{\Pi}_j,&\mbox{ otherwise, }
\end{cases}
\end{equation}
if $I_{\Pi}(n_\Pi)<b_{\Pi}$, otherwise all values remain unchanged.

Fix any $m\leq b$. Observe that in any event the $Q_i$-values change by at most one at any step, and hence it suffices to prove the preservation of the ordering in the case when \eqref{eq:det ord} holds with equality:
\begin{equation}\label{eq:equality}
\sum_{i=m}^b Q_i^A +L^A= \sum_{i=m}^b Q_i^B +L^B.
\end{equation}

We distinguish between two cases depending on whether an item is removed or added.
First suppose that the rule prescribes removal of an item from the $k-$th stack from both ensembles. Observe from \eqref{eq:removal} that the value of $\sum_{i=m}^b Q_i^\Pi +L^\Pi$ changes if and only if $I_\Pi(k)\geq m$. Also, since removal of an item can only decrease the sum, without loss of generality we may assume that $I_B(k)\geq m$, otherwise the right side of \eqref{eq:equality} remains unchanged, and the ordering is trivially preserved. From our initial hypothesis,
\begin{equation}
\sum_{i=m+1}^b Q_i^{A}+L^{A}\leq
\sum_{i=m+1}^b Q_i^{B}+L^{B}.
\end{equation}
This implies
\begin{equation}
\begin{split}
Q_{m}^{A}&=\sum_{i=m}^b Q_i^{A}-\sum_{i=m+1}^b Q_i^{A}\geq \sum_{i=m}^b Q_i^{B}-\sum_{i=m+1}^b Q_i^{B}= Q_m^{B}.
\end{split}
\end{equation}
Also, 
\begin{equation}
\begin{split}
I_B(k)\geq m &\iff Q_m^B\geq N-k+1 \\
&\implies Q_m^A\geq N-k+1 \iff I_A(k)\geq m.
\end{split}
\end{equation}
Therefore the sum $\sum_{i=m}^b Q_i^A +L^A$ also decreases, and the ordering is preserved.

Now suppose that the rule prescribes addition of an item to the respective stacks in both ensembles.
From \eqref{eq:addition} we get that after adding an item, the value of $\sum_{i=m}^b Q_i^\Pi +L^\Pi$ increases only if $I_\Pi(n_{\sss\Pi})\geq~m-1$. As in the previous case, we assume \eqref{eq:equality}, and since adding an item can only increase the concerned sums, we assume that $I_A(n_{A})\geq~m-1$, because otherwise the left side of~\eqref{eq:equality} remains unchanged, and the ordering is trivially preserved. Now from our initial hypothesis we have
\begin{equation}\label{eq:initial}
\sum_{i=m-1}^b Q_i^A +L^A\leq \sum_{i=m-1}^b Q_i^B +L^B.
\end{equation}
Combining \eqref{eq:equality} with \eqref{eq:initial} gives
\begin{equation}\label{eq:prop-stoch-assump3}
\begin{split}
Q_{m-1}^{A}&=\left(\sum_{i=m-1}^b Q_i^{A}+L^A\right)-\left(\sum_{i=m}^b Q_i^{A}+L^A\right)\\
&\leq \left(\sum_{i=m-1}^b Q_i^{B}+L^B\right)-\left(\sum_{i=m}^b Q_i^{B}+L^B\right)= Q_{m-1}^{B}.
\end{split}
\end{equation}
Observe that
\begin{equation}
\begin{split}
I_A(n_{A})\geq m-1 &\iff Q_{m-1}^A\geq N-n_A+1\implies Q_{m-1}^A\geq N-n_B+1\\
&\implies Q_{m-1}^B\geq N-n_B+1\iff I_B(n_{B})\geq m-1.
\end{split}
\end{equation}
Hence, the value of $\sum_{i=m}^b Q_i^B +L^B$ also increases, and the ordering is preserved. 
\end{proof}

\subsection{Stochastic ordering}\label{ssec:stoch-ord}
We now use the deterministic ordering established in Proposition~\ref{prop:det ord} in conjunction with the S-coupling construction to prove a stochastic comparison between the JSQ$(d(N))$ scheme, a specific scheme from the class CJSQ$(n(N))$ and the ordinary JSQ policy.
As described earlier, the class CJSQ$(n(N))$ contains all schemes that assign incoming tasks by some rule to any of the $n(N)+1$ lowest ordered servers. 
Observe that when $n(N)=0$, the class contains only the ordinary  JSQ policy. Also, if $n^{(1)}(N)< n^{(2)}(N)$, then CJSQ$(n^{(1)}(N))\subset$ CJSQ$(n^{(2)}(N)).$ 
Let MJSQ$(n(N))$ be a particular scheme that always assigns incoming tasks to precisely the $(n(N)+1)^{\mathrm{th}}$ ordered server. 
Notice that this scheme is effectively the JSQ policy when the system always maintains $n(N)$ idle servers, or equivalently, uses only $N-n(N)$ servers, and MJSQ($n(N)$) $\in$ CJSQ($n(N)$). 
For brevity, we suppress $n(N)$ in the notation for the remainder of this subsection.

We call any two systems \emph{S-coupled}, if they have synchronized arrival clocks and departure clocks of the $k^{\mathrm{th}}$ longest queue, for $1\leq k\leq N$ (`S' in the name of the coupling stands for `Server').
Consider three S-coupled systems following respectively the JSQ policy, any scheme from the class CJSQ, and the MJSQ scheme. 
Recall that $Q^\Pi_i(t)$ is the number of servers with at least $i$ tasks at time $t$ and $L^\Pi(t)$ is the total number of lost tasks up to time $t$, for the schemes $\Pi=$ JSQ, CJSQ, MJSQ. 
The following proposition provides a stochastic ordering for any scheme in the class CJSQ with respect to the ordinary JSQ policy and the MJSQ scheme.
\begin{proposition}\label{prop:stoch-ord}
For any fixed $m\geq 1$,
\begin{enumerate}[{\normalfont(i)}] 
\item\label{item:jsq-cjsq} $
\left\{\sum_{i=m}^bQ_i^{\JSQ}(t)+L^{\JSQ}(t)\right\}_{t\geq 0}\leq_{\mathrm{st}}
\left\{\sum_{i=m}^bQ_i^{\CJSQ}(t)+L^{\CJSQ}(t)\right\}_{t\geq 0},
$
\item\label{item:cjsq-mjsq} $\left\{\sum_{i=m}^bQ_i^{\CJSQ}(t)+L^{\CJSQ}(t)\right\}_{t\geq 0}\leq_{\mathrm{st}}
\left\{\sum_{i=m}^bQ_i^{\MJSQ}(t)+L^{\MJSQ}(t)\right\}_{t\geq 0},$
\end{enumerate}
provided the inequalities hold at time $t=0$.
\end{proposition}
The above proposition has the following immediate corollary, which will be used to prove bounds on the fluid and the diffusion scale.
\begin{corollary}\label{cor:bound}
In the joint probability space constructed by the S-coupling of the three systems under respectively \emph{JSQ}, \emph{MJSQ}, and any scheme from the class \emph{CJSQ}, the following ordering is preserved almost surely throughout the sample path: for any fixed $m\geq 1$
\begin{enumerate}[{\normalfont(i)}]
\item $Q_m^{\CJSQ}(t)\geq \sum_{i=m}^{b}Q_i^{\JSQ}(t)-\sum_{i=m+1}^{b}Q_i^{\MJSQ}(t)+L^{\JSQ}(t)-L^{\MJSQ}(t)$,
\item $Q_m^{\CJSQ}(t)
\leq \sum_{i=m}^{b}Q_i^{\MJSQ}(t)-\sum_{i=m+1}^{b}Q_i^{\JSQ}(t)+L^{\MJSQ}(t)-L^{\JSQ}(t),$
\end{enumerate}
provided the inequalities hold at time $t=0$.
\end{corollary}

\begin{proof}[Proof of Proposition~\ref{prop:stoch-ord}]
We first S-couple the concerned systems. 
Let us say that an incoming task is assigned to the $n_\Pi^{\mathrm{th}}$ ordered server under scheme $\Pi$, $\Pi$= JSQ, CJSQ, MJSQ. Then observe that, under the S-coupling, almost surely, 
$n_\JSQ\leq n_\CJSQ\leq n_\MJSQ.$ Therefore, Proposition~\ref{prop:det ord} ensures that in the probability space constructed through the S-coupling, the ordering is preserved almost surely throughout the sample path. 
\end{proof}
\begin{remark}
\normalfont
Note that $\sum_{i=1}^b\min\big\{Q_i,k\big\}$ represents the aggregate size of the rightmost $k$ stacks, i.e., the $k$ longest queues.
Using this observation, the stochastic majorization property
of the JSQ policy as stated in \cite{towsley, Towsley95,Towsley1992}
can be shown following similar arguments as in the proof of Proposition~\ref{prop:stoch-ord}.
Conversely, the stochastic ordering between the JSQ policy and the MJSQ scheme presented in Proposition~\ref{prop:stoch-ord} can also be derived from the weak majorization arguments developed in \cite{towsley, Towsley95,Towsley1992}. But it is only through the stack arguments developed in the previous subsection that we could extend the results to compare any scheme from the class CJSQ with the scheme MJSQ as well
as in Proposition~\ref{prop:stoch-ord}~\eqref{item:cjsq-mjsq}.
\end{remark}

To analyze the JSQ$(d(N))$ scheme, we need a further stochastic comparison argument.
Consider two S-coupled systems following schemes $\Pi_1$ and $\Pi_2$.
Fix a specific arrival epoch, and let the arriving task join the $n_{\Pi_i}^{\mathrm{th}}$ ordered server in the $i^{\mathrm{th}}$ system following scheme $\Pi_i$, $i=1,2$ (ties can be broken arbitrarily in both systems). 
We say that at a specific arrival epoch the two systems \emph{differ in decision}, if $n_{\Pi_1}\neq n_{\Pi_2}$,
and denote by $\Delta_{\Pi_1,\Pi_2}(t)$ the cumulative number of times the two systems differ in decision up to time~$t$.

\begin{proposition}\label{prop:stoch-ord2}
For two S-coupled systems under schemes $\Pi_1$ and $\Pi_2$ the following inequality is preserved almost surely
\begin{equation}\label{eq:stoch-ord2}
\sum_{i=1}^b |Q_i^{\Pi_1}(t)-Q_i^{\Pi_2}(t)|\leq 2\Delta_{\Pi_1,\Pi_2}(t)\qquad\forall\ t\geq 0,
\end{equation}
provided the two systems start from the same occupancy state at $t=0$, i.e., $Q_i^{\Pi_1}(0)= Q_i^{\Pi_2}(0)$ for all $i=1,2,\ldots, b$.
\end{proposition}

\begin{proof}
We will again use forward induction on the event times of arrivals and departures. Let the inequality~\eqref{eq:stoch-ord2} hold at time epoch $t_0$, and let $t_1$ be the next event time.
We distinguish between two cases, depending on whether $t_1$ is an arrival epoch or a departure epoch. 

If $t_1$ is an arrival epoch and the systems differ in decision, then observe that the left side of \eqref{eq:stoch-ord2} can only increase by two. In this case, the right side also increases by two, and the inequality is preserved.
Therefore, it is enough to prove that the left side of \eqref{eq:stoch-ord2} remains unchanged if the two systems do not differ in decision. In that case, assume that both $\Pi_1$ and $\Pi_2$ assign the arriving task to the $k^{\mathrm{th}}$ ordered server.
Recall from the proof of Proposition~\ref{prop:det ord} the definition of $I_\Pi$ for some scheme $\Pi$. 
If $I_{\sss\Pi_1}(k)=I_{\sss\Pi_2}(k)$, then the left side of \eqref{eq:stoch-ord2} clearly remains unchanged. Now, without loss of generality, assume $I_{\sss\Pi_1}(k)<I_{\sss\Pi_2}(k)$. Therefore, 
$$Q_{\sss I_{\Pi_1}(k)+1}^{\Pi_1}(t_0)< Q_{\sss I_{\Pi_1}(k)+1}^{\Pi_2}(t_0).$$
After an arrival, the $ (I_{\Pi_1}(k)+1)$-th term in the left side of \eqref{eq:stoch-ord2} decreases by one, and the $ (I_{\Pi_2}(k)+1)$-th term may increases by at most one. Thus the inequality is preserved.

If $t_1$ is a departure epoch, then due to the S-coupling, without loss of generality, assume that a potential departure occurs from the $k^{\mathrm{th}}$ ordered server. Also note that a departure in either of the two systems can change at most one of the $Q_i$-values. 
If at time epoch $t_0$, $I_{\Pi_1}(k)=I_{\Pi_2}(k)=i$, then both $Q_{\sss i}^{\Pi_1}$ and $Q_{\sss i}^{\Pi_2}$ decrease by one, and hence the left side of \eqref{eq:stoch-ord2} does not change. 
Otherwise, without loss of generality assume $I_{\Pi_1}(k)<I_{\Pi_2}(k).$ Then observe that 
$$Q_{\sss I_{\Pi_2}(k)}^{\Pi_1}(t_0)<Q_{\sss I_{\Pi_2}(k)}^{\Pi_2}(t_0).$$ 
Furthermore, after the departure, $Q_{\sss I_{\Pi_1}(k)}^{\Pi_1}$ may decrease at most by one. Therefore $|Q_{\sss I_{\Pi_1}(k)}^{\Pi_1}- Q_{\sss I_{\Pi_1}(k)}^{\Pi_2}|$ may increase at most by one, and $Q_{\sss I_{\Pi_2}(k)}^{\Pi_2}$ decreases by one, thus $|Q_{\sss I_{\Pi_2}(k)}^{\Pi_1}- Q_{I_{\sss \Pi_2}(k)}^{\Pi_2}|$ decreases by one. Hence, in total, the left side of \eqref{eq:stoch-ord2} either remains the same or decreases by one.
\end{proof}

\subsection{Comparing the JSQ(d) and CJSQ(n) schemes}
We will now introduce the JSQ$(n,d)$ scheme with $n,d\leq N$, which is an intermediate blend between 
the CJSQ$(n)$ schemes and the JSQ$(d)$ scheme.
The JSQ$(n,d)$ scheme will be seen in a moment to  be a scheme in the CJSQ$(n)$ class. 
It will also be seen to approximate the JSQ$(d)$ scheme closely.
We now specify the JSQ$(d,n)$ scheme. At its first step, just as in the JSQ$(d)$ scheme, it first chooses the shortest of $d$ random candidates but only sends this to that server's queue if it is one of the $n+1$ shortest queues. 
If it is not, then at the second step it picks any of the $n+1$ shortest queues uniformly at random and then sends to that server's queue. 
As was mentioned earlier, by construction, JSQ$(d,n)$ is a scheme in CJSQ$(n)$.

We now consider two S-coupled systems with a JSQ$(d)$ and a JSQ$(n,d)$ scheme.
Assume that at some specific arrival epoch, the incoming task is dispatched to the $k^{\mathrm{th}}$ ordered server in the system under the JSQ($d$) scheme. If $k\in\{1,2,\ldots,n+1\}$, then the system under JSQ$(n,d)$ scheme also assigns the arriving task to the $k^{\mathrm{th}}$ ordered server. 
Otherwise, it dispatches the arriving task uniformly at random among the first $(n+1)$ ordered servers.

In the next proposition we will bound the number of times these two systems differ in decision on any finite time interval.
For any $T\geq 0$, let $A(T)$ and $\Delta(T)$ be the total number of arrivals to the system and the cumulative number of times that the JSQ($d$) scheme and JSQ$(n,d)$ scheme differ in decision up to time $T$.
\begin{proposition}\label{prop:differ}
For any $T\geq 0$, and $M>0,$
\begin{equation}
\Pro{\Delta(T)\geq M\given A(T)}\leq \frac{A(T)}{M}\left(1-\frac{n}{N}\right)^{d}.
\end{equation}
\end{proposition}
\begin{proof}
Observe that at any arrival epoch, the systems under the JSQ$(d)$ scheme and the JSQ$(n,d)$ scheme will differ in decision 
only if none of the $n$ lowest ordered servers gets selected by the JSQ$(d)$ scheme.
Now, at any arrival epoch, the probability that the JSQ($d$) scheme does not select any of the $n$ lowest ordered servers, is given by 
$$p=\left(1-\frac{n}{N}\right)^{d}.$$
Since at each arrival epoch, $d$ servers are selected independently, given $A(T)$, 
$$\Delta(T)\sim \mbox{Bin}(A(T),p).$$
Therefore, for $T\geq 0$, Markov's inequality yields, for any fixed $M>0$,
$$\Pro{\Delta(T)\geq M\given A(T)}\leq \frac{\E{\Delta(T)\given A(T)}}{M}=\frac{A(T)}{M}\left(1-\frac{n}{N}\right)^{d}.$$
\end{proof}

\section{Fluid-Limit Proofs}\label{sec:fluid}
In this section we prove the fluid-limit results for the JSQ$(d(N))$ scheme stated in Theorems~\ref{fluidjsqd} and~\ref{th:batch}.
As mentioned in Subsection~\ref{subsec:strategy}, the fluid limit for the ordinary JSQ policy is provided in Subsection~\ref{ssec:fluidjsq}, and in Subsection~\ref{ssec:equivjsq} we prove a universality result establishing that under the condition that $d(N)\to\infty$ as $N\to\infty$, the fluid limit for the JSQ$(d(N))$ scheme coincides with that for the ordinary JSQ policy.

\subsection{Fluid limit of JSQ}\label{ssec:fluidjsq}
In this section we establish the fluid limit for the ordinary JSQ policy and the interchange of limits result stated in Proposition~\ref{th:interchange}.
In the proof we will leverage the time scale separation technique developed in~\cite{HK94}, suitably extended to an infinite-dimensional space. 
As mentioned in the introduction, to the best of our knowledge, this is the first time the transient fluid limit of the ordinary JSQ policy is rigorously established.
We also observe that in order to exploit the coupling framework in Section~\ref{ssec:stoch-ord} and in particular Proposition~\ref{prop:stoch-ord}, we need convergence of tail-sums.
Thus we need to establish the fluid convergence result with respect to the $\ell_1$ topology, which makes the analysis technically challenging.

To leverage the time scale separation technique, note that the rate at which incoming tasks join a server
with $i$ active tasks is determined only by the process $\ZZ^N(\cdot)=(Z_1^N(\cdot),\ldots,Z_b^N(\cdot))$, where $Z_i^N(t)=N-Q_i^N(t)$, $i=1,\ldots,b$, represents the number of servers with fewer than $i$ tasks at time $t$.
Furthermore, the dynamics of the $\ZZ^N(\cdot)$ process can be described as
\begin{equation}\label{eq:prelimit-slowprocess}
\ZZ^N \rightarrow 
\begin{cases}
\ZZ^N+e_i& \quad\mbox{ at rate }\quad N(q_i-q_{i+1}),\\
\ZZ^N-e_i& \quad\mbox{ at rate }\quad N\lambda \ind{\ZZ^N\in\mathcal{R}_{i}},
\end{cases}
\end{equation}
where $e_i$ is the $i^{\mathrm{th}}$ unit vector, and 
\begin{equation}\label{eq:partition}
\mathcal{R}_i := \big\{(z_1,z_2,\ldots, z_b): z_1=\ldots=z_{i-1}=0<z_i\leq z_{i+1}\leq\ldots\leq z_b\big\}\in\mathcal{G},
\end{equation}
$i=1,2,\ldots,b$,
with the convention that $Q^N_{b+1}$ is always taken to be zero, if $b<\infty$.
Observe that  in any time interval $[t,t+\varepsilon]$ of length $\varepsilon>0$, the $\ZZ^N(\cdot)$ process experiences $O(\varepsilon N)$ events (arrivals and departures), 
while the $\qq^N(\cdot)$ process can change by only $O(\varepsilon)$ amount.
In other words, loosely speaking, around a `small' neighborhood of time $t$, the $q_i(t)$'s are constants, while as $N\to \infty$, the process $\ZZ^N(\cdot)$ behaves as a {\em time-scaled} version of the following process:
\begin{equation}
\ZZ_{\qq(t)} \rightarrow 
\begin{cases}
\ZZ_{\qq(t)}+e_i& \quad\mbox{ at rate }\quad q_i(t)-q_{i+1}(t),\\
\ZZ_{\qq(t)}-e_i& \quad\mbox{ at rate }\quad \lambda \ind{\ZZ_{\qq(t)}\in\mathcal{R}_{i}}.
\end{cases}
\end{equation}
Therefore, the $\ZZ^N(\cdot)$ process evolves on a much faster time scale than
the $\qq^N(\cdot)$ process.
As a result, in the limit as $N\to\infty$, at each time point $t$, the $\ZZ^N(\cdot)$ process
achieves stationarity depending on the instantaneous value of the $\qq^N(\cdot)$ process, i.e., a separation of time scales takes place. 
In order to establish the time-scale separation and the fluid limit results, we first write the evolution of the occupancy states in terms of a suitable random measure (see~\eqref{eq:martingale rep assumption 2-2}) and establish in Proposition~\ref{prop:rel compactness} that the sequence of joint occupancy process and the random measure is relatively compact. 
We also characterize the limit of any convergence subsequence, where we invoke analogous arguments as used in the proofs of~\cite[Lemma 2]{HK94} and \cite[Theorem 3]{HK94} to complete the proof of the separation of time scales.
The proof of the fluid limit result is then completed by establishing uniqueness of the instantaneous stationary distribution achieved by the fast process, given any fluid-scaled occupancy state.

Denote by $\bZ_+$ the one-point compactification of the set of nonnegative integers~$\Z_+$, i.e., $\bZ_+=\Z_+\cup\{\infty\}$.
Equip $\bZ_+$ with the order topology. Denote $G=\bZ_+^b$ equipped with product topology, and with the Borel $\sigma$-algebra $\mathcal{G}$.
Let us consider the $G$-valued process $\ZZ^N(s):=\big(Z_i^N(s)\big)_{i\geq 1}$ as introduced above.
Note that for the ordinary JSQ policy, the probability that a task arriving at (say) $t_k$ is assigned to 
some server with $i$ active tasks is given by $p_{i-1}^N(\QQ^N(t_k-))=\ind{\ZZ^N(t_k-)\in\mathcal{R}_{i}}$, where 
$\mathcal{R}_i$ is as in~\eqref{eq:partition}.
We prove the following fluid-limit result for the ordinary JSQ policy.
Recall the definition of $m(\qq)$ in Subsection~\ref{sec:fluidresult}. If $m(\qq)>0$, then define
\begin{equation}\label{eq:fluid-gen}
p_{i}(\qq)=
\begin{cases}
\min\big\{(1-q_{m(\qq)+1})/\lambda,1\big\} & \quad\mbox{ for }\quad i=m(\qq)-1,\\
1 - p_{\sss m(\qq) - 1}(\qq) & \quad\mbox{ for }\quad i=m(\qq),\\
0&\quad \mbox{ otherwise,}
\end{cases}
\end{equation}
and else, define $p_0(\qq) = 1$ and $p_i(\qq) = 0$ for all $i = 1,\ldots,b$.
\begin{theorem}[{Fluid limit of JSQ}]
\label{th:genfluid}
Assume $\qq^N(0)\pto \qq^\infty$ in $\mathcal{S}$ and $\lambda(N)/N\to\lambda>0$ as $N\to\infty$. Then any subsequence of the sequence of processes $\big\{\qq^N(t)\big\}_{t\geq 0}$ for the ordinary JSQ policy has a further subsequence that converges weakly with respect to the Skorohod $J_1$ topology to the limit $\{\qq(t)\}_{t\geq 0}$ satisfying the following system of integral equations
\begin{equation}\label{eq:fluidfinal}
q_i(t) = q_i(0)+\lambda \int_0^t p_{i-1}(\qq(s))\dif s - \int_0^t (q_i(s)-q_{i+1}(s))\dif s, \quad i=1,2,\ldots,b,
\end{equation}
where $\qq(0) = \qq^\infty$ and the coefficients $p_i(\cdot)$ are as
defined in~\eqref{eq:fluid-gen}.
\end{theorem}
 The rest of this section will be devoted in the proof of Theorem~\ref{th:genfluid}.
First we construct the martingale representation of the occupancy state process $\QQ^N(\cdot)$.
Note that the component $Q_i^N(t)$, satisfies the identity relation
\begin{align}\label{eq:recursion}
Q_i^N(t)=Q_i^N(0)+A_i^N(t)-D_i^N(t),&\quad\mbox{ for }\quad i=1,\ldots, b,
\end{align}
where
\begin{align*}
A_i^N(t)&=\mbox{ number of arrivals during $[0,t]$ to some server with }i-1\mbox{ active tasks,} \\
D_i^N(t)&=\mbox{ number of departures during $[0,t]$ from some server with }i\mbox{ active tasks}.
\end{align*}
We can express $A^N_i(t)$ and $D_i^N(t)$ as
\begin{align*}
A^N_i(t) &=  \mathcal{N}_{A,i}\left(\lambda(N)\int_0^t p_{i-1}^N(\QQ^N(s))\dif s\right),\\
D_i^N(t) &=  \mathcal{N}_{D,i}\left(\int_0^t(Q^N_i(s)-Q^N_{i+1}(s))\dif s\right),
\end{align*}
where $\mathcal{N}_{A,i}$ and $\mathcal{N}_{D,i}$ are mutually independent unit-rate Poisson processes, $i=1,2,\ldots,b$.
Define the following sigma fields
\begin{align*}
\mathcal{A}^N_i(t)&:= \sigma\left(A^N_i(s): 0\leq s\leq t\right),\\
\mathcal{D}_i^N(t)&:= \sigma\left(D_i^N(s): 0\leq s\leq t\right),\mbox{ for }i = 1,\ldots,b,
\end{align*}
and the filtration $\mathbf{F}^N\equiv\big\{\mathcal{F}^N_t:t\geq 0\big\}$ with
\begin{equation}\label{eq:filtration}
\mathcal{F}^N_t:=\bigvee_{i=1}^{\infty} [\mathcal{A}_i^N(t)\vee \mathcal{D}_i^N(t)]
\end{equation}
augmented by all the null sets. 
Now we have the following martingale decomposition from the random time change of a unit-rate Poisson process result in \cite[Lemma~3.2]{PTRW07}.

\begin{proposition}[Martingale decomposition]
\label{prop:mart-rep}
The following are $\mathbf{F}^N$-martingales, for $i\geq 1$:
\begin{equation}\label{eq:martingales}
\begin{split}
M_{A,i}^N(t)&:=  \mathcal{N}_{A,i}\left(\lambda(N)\int_0^t p_{i-1}^N(\QQ^N(s))\dif s\right)-\lambda(N)\int_0^t p_{i-1}^N(\mathbf{Q}^N(s)) \dif s,\\
M_{D,i}^N(t)&:=\mathcal{N}_{D,i}\left(\int_0^t (Q^N_i(s)-Q^N_{i+1}(s))\dif s\right)-\int_0^t (Q^N_i(s)-Q^N_{i+1}(s))\dif s,
\end{split}
\end{equation}
with respective compensator and predictable quadratic variation processes given by
\begin{align*}
\langle M_{A,i}^N\rangle(t)&:= \lambda(N)\int_0^t p_{i-1}^N(\mathbf{Q}^N(s-))\dif s,\\
\langle M_{D,i}^N\rangle(t)&:=\int_0^t (Q^N_i(s)-Q^N_{i+1}(s))\dif s.
\end{align*}
\end{proposition}
Therefore, finally we have the following martingale representation of the $N^{\mathrm{th}}$ process:
\begin{equation}\label{eq:mart-unscaled}
\begin{split}
Q_i^N(t)&=Q_i^N(0)+\lambda(N)\int_0^t p_{i-1}^N(\mathbf{Q}^N(s))\dif s
-\int_0^t (Q^N_i(s)-Q^N_{i+1}(s))\dif s \\
&\hspace{4cm}+(M_{A,i}^N(t)-M_{D,i}^N(t)),\quad t\geq 0,\quad i= 1,\ldots,b.
\end{split}
\end{equation}
In the proposition below, we prove that the martingale part vanishes in $\ell_1$ when scaled by~$N$. 

\begin{proposition}[Convergence of martingales]
\label{prop:mart zero1}
$$\left\{\frac{1}{N}\sum_{i\geq 1}(|M_{A,i}^N(t)|+|M_{D,i}^N(t)|)\right\}_{t\geq 0}\dto \big\{m(t)\big\}_{t\geq 0}\equiv 0.$$
\end{proposition}
\begin{proof}
The proof follows using the same line of arguments as in the proof of~\cite[Theorem 3.13]{Mitzenmacher1996}, and hence is sketched only briefly for the sake of completeness.
Fix any $T\geq 0$, and observe that 
\begin{align}
&\lim_{N\to\infty}\sup_{t\in[0,T]}\frac{1}{N}\sum_{i\geq 1}|M_{A,i}^N(t)|\label{subeq:4-8} \\
&=
\lim_{N\to\infty}\sup_{t\in[0,T]}\frac{1}{N}\left(\sum_{i\geq 1} \left|\mathcal{N}_{A,i}\left(\lambda(N)\int_0^t p_{i-1}^N(\QQ^N(s))\dif s\right)-\lambda(N)\int_0^t p_{i-1}^N(\QQ^N(s))\dif s\right| \right)\label{subeq:4-9}\\
&\leq \lim_{N\to\infty}\frac{1}{N}\sum_{i\geq 1} \mathcal{N}_{A,i}\left(\lambda(N)\int_0^T p_{i-1}^N(\QQ^N(s))\dif s\right)+\lambda T\label{subeq:4-10}.
\end{align} 
Since $N^{-1}\sum_{i\geq 1} \lambda(N)\int_0^t p_{i-1}^N(\QQ^N(s))\dif s\to \lambda t<\infty,$ the $\lim_{N\to\infty}$ and $\sum_{i\geq 1}$ above can be interchanged in~\eqref{subeq:4-10}, and hence in~\eqref{subeq:4-8}.
Now for each $i\geq 1$, from Doob's inequality~\cite[Theorem 1.9.1.3]{LS89}, we have for any $\epsilon>0,$
\begin{align*}
\Pro{\sup_{t\in[0,T]}\frac{1}{N}M_{A,i}^N(t)\geq \epsilon}&=\Pro{\sup_{t\in[0,T]}M_{A,i}^N(t)\geq N\epsilon}
\leq \frac{1}{N^2\epsilon^2}\E{\langle M_{A,i}^N\rangle (T)}\\
&\leq \frac{1}{N\epsilon^2}\int_0^T p_{i-1}(\mathbf{Q}^N(s-))\lambda(N)\dif s
\leq \frac{\lambda T}{N\epsilon^2}\to 0,\mbox{ as }N\to\infty.
\end{align*}
Thus $\sup_{t\in[0,T]}N^{-1}M_{A,i}^N(t)\pto 0$, and hence, 
$\sup_{t\in[0,T]}N^{-1}\sum_{i\geq 1}|M_{A,i}^N(t)|\pto 0.$
Using similar arguments as above, we can also show that $\sup_{t\in[0,T]}N^{-1}\sum_{i\geq 1}|M_{D,i}^N(t)|\pto 0,$ and the proof is complete.
\end{proof}

Now we prove the relative compactness of the sequence of fluid-scaled processes.
Recall that we denote all the fluid-scaled quantities by their respective small letters, e.g.~$\mathbf{q}^N(t):=\mathbf{Q}^N(t)/N$, componentwise, i.e., $q_i^N(t):=Q_i^N(t)/N$ for $i\geq 1$. Therefore the martingale representation in~\eqref{eq:mart-unscaled}, can be written as
\begin{equation}\label{eq:mart1}
\begin{split}
q_i^N(t)&=q_i^N(0)+\frac{\lambda(N)}{N}\int_0^t p_{i-1}^N(\mathbf{Q}^N(s))\dif s
-\int_0^t (q^N_i(s)-q^N_{i+1}(s))\dif s\\
&\hspace{4cm} +\frac{1}{N}(M_{A,i}^N(t)-M_{D,i}^N(t)),\quad i=1,2,\ldots, b,
\end{split}
\end{equation}
or equivalently,
\begin{equation}\label{eq:martingale rep assumption 2}
\begin{split}
q_i^N(t)&=q_i^N(0)+\frac{\lambda(N)}{N}\int_0^t \ind{\ZZ^N(s)\in\mathcal{R}_{i}}\dif s
-\int_0^t (q^N_i(s)-q^N_{i+1}(s))\dif s \\
&\hspace{4cm}+\frac{1}{N}(M_{A,i}^N(t)-M_{D,i}^N(t)),\quad i=1,2,\ldots, b.
\end{split}
\end{equation}
Now, we consider the Markov process $(\qq^N,\ZZ^N)(\cdot)$ defined on $\mathcal{S}\times G$. 
Define a random measure $\alpha^N$ on the measurable space $([0,\infty)\times G, \mathcal{C}\otimes\mathcal{G})$, when $[0,\infty)$ is endowed with Borel sigma algebra $\mathcal{C}$, by
\begin{equation}
\alpha^N(A_1\times A_2):=\int_{A_1} \ind{\ZZ^N(s)\in A_2}\dif s,
\end{equation}
for $A_1\in\mathcal{C}$ and $A_2\in\mathcal{G}$. 
Then the representation in \eqref{eq:martingale rep assumption 2} can be written in terms of the random measure as
\begin{equation}\label{eq:martingale rep assumption 2-2}
\begin{split}
q_i^N(t)&=q_i^N(0)+\lambda\alpha^N([0,t]\times\mathcal{R}_i)
-\int_0^t (q^N_i(s)-q^N_{i+1}(s))\dif s \\
&\hspace{4cm} +\frac{1}{N}(M_{A,i}^N(t)-M_{D,i}^N(t)),\quad i=1,2,\ldots, b.
\end{split}
\end{equation}
Let $\mathfrak{L}$ denote the space of all measures on $[0,\infty)\times G$ satisfying $\gamma([0,t],G)= t$, endowed with the topology corresponding to weak convergence of measures restricted to $[0,t]\times G$ for each $t$.
\begin{proposition}[Relative compactness]
\label{prop:rel compactness}
Assume $\mathbf{q}^N(0)\dto\qq^\infty\in \mathcal{S}$ as $N\to\infty$, then $\big\{(\mathbf{q}^N(\cdot),\alpha^N)\big\}_{N\geq 1}$ is a relatively compact sequence in $D_{\mathcal{S}}[0,\infty)\times\mathfrak{L}$ and the limit $(\mathbf{q}(\cdot),\alpha)$ of any convergent subsequence satisfies
\begin{equation}\label{eq:rel compact}
q_i(t)=q_i^\infty+\lambda \alpha([0,t]\times\mathcal{R}_i) -\int_0^t (q_i(s)-q_{i+1}(s))\dif s,\quad i=1,2,\ldots, b.
\end{equation}
\end{proposition}

To prove Proposition~\ref{prop:rel compactness}, we will verify the  relative compactness conditions from~\cite{EK2009}. 
Let $(E,r)$ be a complete and separable metric space. For any $x\in D_E[0,\infty)$, $\delta>0$ and $T>0$, define
\begin{equation}\label{eq:mod-continuity}
w'(x,\delta,T)=\inf_{\{t_i\}}\max_i\sup_{s,t\in[t_{i-1},t_i)}r(x(s),x(t)),
\end{equation}
where $\{t_i\}$ ranges over all partitions of the form $0=t_0<t_1<\ldots<t_{n-1}<T\leq t_n$ with $\min_{1\leq i\leq n}(t_i-t_{i-1})>\delta$ and $n\geq 1$.
 Below we state the conditions for the sake of completeness.
\begin{theorem}\label{th:from EK}
\begin{normalfont}
\cite[Corollary~3.7.4]{EK2009}
\end{normalfont}
Let $(E,r)$ be complete and separable, and let $\big\{X_n\big\}_{n\geq 1}$ be a family of processes with sample paths in $D_E[0,\infty)$. Then $\big\{X_n\big\}_{n\geq 1}$ is relatively compact if and only if the following two conditions hold:
\begin{enumerate}[{\normalfont (a)}]
\item For every $\eta>0$ and rational $t\geq 0$, there exists a compact set $\Gamma_{\eta, t}\subset E$ such that $$\varliminf_{n\to\infty}\Pro{X_n(t)\in\Gamma_{\eta, t}}\geq 1-\eta.$$
\item For every $\eta>0$ and $T>0$, there exists $\delta>0$ such that
$$\varlimsup_{n\to\infty}\Pro{w'(X_n,\delta, T)\geq\eta}\leq\eta.$$
\end{enumerate}
\end{theorem}

In order to prove the relative compactness, we will need the next three lemmas: Lemma~\ref{lem:technicalball} characterizes the relatively compact subsets of $\mathcal{S}$, Lemma~\ref{lem:tightcond} provides a necessary and sufficient criterion for a sequence of $\ell_1$-valued random variables to be tight, and Lemma~\ref{lem:trivialbound} is needed to ensure that at all finite times $t$, the occupancy state process lies in some compact set (possibly depending upon $t$).
\begin{lemma}[Compact subsets of $\mathcal{S}$]
\label{lem:technicalball}
Assume $b=\infty.$
A set $K\subseteq \mathcal{S}$ is relatively compact in $\mathcal{S}$ with respect to $\ell_1$ topology if and only if
\begin{equation}
\label{eq:compactcond}
\lim_{k\to\infty} \sup_{\xx\in K} \sum_{i=k}^\infty x_i = 0.
\end{equation}
\end{lemma} 
\begin{proof}
For the if part, fix any $K\subseteq \mathcal{S}$ satisfying~\eqref{eq:compactcond}.
We will show that 
%
 any sequence $\big\{\xx^n\big\}_{n\geq 1}$ in $K$ 
has a Cauchy subsequence.
Since the $\ell_1$ space is complete, this will then imply that $\big\{\xx^n\big\}_{n\geq 1}$ has a convergent subsequence with the limit in $\overline{K}$,
which will complete the proof.

To show the existence of a Cauchy sequence, fix any $\varepsilon>0$, and choose $k\geq 1$ (depending on $\varepsilon$) such that
\begin{equation}\label{eq:uppertailtech}
\sum_{i\geq k}|x_i^n|<\frac{\varepsilon}{4}\qquad\forall\ n\geq 1.
\end{equation}
Now observe that the set of first coordinates $\big\{x_1^n\big\}_{n\geq 1}$ is a sequence in $[0,1]$, and hence has a convergent subsequence. Along that subsequence, the set of the second coordinates has a further convergent subsequence. Proceeding this way, we can get a subsequence along which the first $k-1$ coordinates converge. Therefore, depending upon $\varepsilon$, an $N'\in\N$ can be chosen, such that 
\begin{equation}\label{eq:lowertailtech}
\sum_{i<k}|x_i^n-x_i^m|<\frac{\varepsilon}{2}\qquad\forall\ m,n \geq N'.
\end{equation}
Therefore,~\eqref{eq:uppertailtech} and~\eqref{eq:lowertailtech} yields for all $n\geq \max\big\{N,N'\big\}$,
\begin{align*}
\norm{\xx^n-\xx^m}
&=\sum_{i\geq 1}|x_i^n-x_i^m|
\leq \sum_{i< k}|x_i^n-x_i^m| +\sum_{i\geq k}|x_i^n-x_i^m|\\
&\leq \sum_{i< k}|x_i^n-x_i^m| +\sum_{i\geq k}x_i^n+\sum_{i\geq k}x_i^m
<\varepsilon
\end{align*}
along the above suitably constructed subsequence. 
Now that the limit point is in $\mathcal{S}$ follows from the completeness of $\ell_1$ space and the fact that $\mathcal{S}$ is a closed subset of $\ell_1$.
Indeed, since the $\ell_1$ topology is finer than the product topology, any set that is closed with respect to the product topology is closed with respect to the $\ell_1$ topology, and observe that $\mathcal{S}$ is closed with respect to the product topology.

For the only if part, let $K\subseteq \mathcal{S}$ be relatively compact, and on the contrary, assume that there exists an $\varepsilon>0,$ such that
\begin{equation}
\lim_{k\to\infty} \sup_{\xx\in K} \sum_{i=k}^\infty x_i \geq \varepsilon.
\end{equation}
Therefore, for each $k\geq 1$, there exists $\xx^{(k)}\in K$, such that $\sum_{i=k}^\infty x^{(k)}_i \geq \varepsilon/2$.
Consider any limit point $\xx^*$ of the sequence $\big\{\xx^{(k)}\big\}_{k\geq 1}$, and note that $\sum_{i=j}^\infty x^*_i \geq \varepsilon/2$ for all $j\geq 1.$ This contradicts that $\xx^*\in\ell_1$, and the proof is complete.
\end{proof}

\begin{lemma}[Criterion for $\ell_1$-tightness]
\label{lem:tightcond}
Let $\big\{\XX^N\big\}_{N\geq 1}$ be a sequence of random variables in $\mathcal{S}$. Then the following are equivalent: 
\begin{enumerate}[{\normalfont (i)}]
\item $\big\{\XX^N\big\}_{N\geq 1}$ is tight with respect to product topology, and
for all $\varepsilon>0,$
\begin{equation}\label{eq:smalltail}
\lim_{k\to\infty}\varlimsup_{N\to\infty}\mathbbm{P}\Big(\sum_{i\geq k}X_i^N>\varepsilon\Big) = 0.
\end{equation}
\item $\big\{\XX^N\big\}_{N\geq 1}$ is tight with respect to $\ell_1$ topology.
\end{enumerate}
\end{lemma}
\begin{proof}
To prove (i)$\implies$(ii),
for any $\varepsilon>0$, we will construct a relatively compact set compact set $K(\varepsilon)$ such that
$$\Pro{\XX^N\notin \overline{K(\varepsilon)}}<\varepsilon\quad\mbox{for all } N.$$
Observe from~\eqref{eq:smalltail} that for all $\varepsilon>0$, there exists 
an $r(\varepsilon)\geq 1$, such that
$$\varlimsup_{N\to\infty}\mathbbm{P}\Big(\sum_{i\geq r(\varepsilon)}X_i^N>\varepsilon\Big) <\varepsilon,$$
and with it an $N(\varepsilon)\geq 1$, such that 
$$\mathbbm{P}\Big(\sum_{i\geq k(\varepsilon)}X_i^N>\varepsilon\Big)<\varepsilon\quad\mbox{for all } N> N(\varepsilon).$$
Furthermore, since $\big\{\XX^1, \XX^2,\ldots, \XX^{N(\varepsilon)}\big\}$ is a finite set of $\ell_1$-valued random variables, there exists 
$k(\varepsilon)\geq r(\varepsilon)$, such that 
$$\mathbbm{P}\Big(\sum_{i\geq k(\varepsilon)}X_i^N>\varepsilon\Big)<\varepsilon\quad\mbox{for all } N.$$
Thus, there exists an increasing sequence $\big\{k(n)\big\}_{n\geq 1}$
such that 
$$\mathbbm{P}\Big(\sum_{i\geq k(n)}X_i^N>\frac{\varepsilon}{2^n}\Big)<\frac{\varepsilon}{2^n}\quad\mbox{for all } N.$$
Define the set $K(\varepsilon)$ as
$$K(\varepsilon):= \Big\{\xx\in \mathcal{S}:\sum_{i\geq k(n)}x_i\leq \frac{\varepsilon}{2^n}\quad\mbox{for all}\quad n\geq 1\Big\}.$$
Due to Lemma~\ref{lem:technicalball}, we know $K(\varepsilon)$ is relatively compact in $\ell_1$.
Also,
\begin{align*}
\mathbbm{P}\big(\XX^N\notin \overline{K(\varepsilon)}\big)&= \mathbbm{P}\Big(\bigcup_{n\geq 1}\Big\{\sum_{i\geq k(n)}X_i^N>\frac{\varepsilon}{2^n}\Big\}\Big)
\leq \sum_{n\geq 1}\mathbbm{P}\Big(\sum_{i\geq k(n)}X_i^N>\frac{\varepsilon}{2^n}\Big)<\varepsilon.
\end{align*}
To prove (ii)$\implies$(i), first observe that a sequence of random variables is tight with respect to the $\ell_1$ topology implies that it must be tight with respect to the product topology. 
Now assume on the contrary to~\eqref{eq:smalltail}, that there exists $\varepsilon>0$, such that 
\begin{equation}\label{eq:heavytail}
\lim_{k\to\infty}\varlimsup_{N\to\infty}\mathbbm{P}\Big(\sum_{i\geq k}X_i^N>\varepsilon\Big)>\varepsilon.
\end{equation}
Since $\big\{\XX^N\big\}_{N\geq 1}$ is tight with respect to $\ell_1$ topology, take any convergent subsequence $\big\{\XX^{N(n)}\big\}_{n\geq 1}$ with $\XX^*$ being a random variable following the limiting measure.
In that case, observe that~\eqref{eq:heavytail} implies $\mathbbm{P}\big(\sum_{i\geq k}X_i^*>\varepsilon/2\big)>\varepsilon$ for all $k\geq 1$, which leads to a contradiction since $\XX^*$ is an $\ell_1$-valued random variable.
\end{proof}

\begin{lemma}
\label{lem:trivialbound}
For any $\qq\in \mathcal{S}$, assume that $\qq^N(0)\dto \qq^\infty$, as $N\to\infty$.
Then for any $t\geq 0$, there exists $M(t, \qq^\infty)\geq 1$, such that under the JSQ policy, with probability tending to one as $N\to\infty$, no arriving task is assigned to a server with $M(t,\qq^\infty)-1$ active tasks up to time $t$.
\end{lemma}
\begin{proof}
Let $A^N(t)$ be the cumulative number of tasks arriving up to time $t$.
Since the arrival rate is $\lambda(N)$, and $\lambda(N)/N\to\lambda$, as $N\to\infty,$ for any $\varepsilon>0$, 
$$\Pro{A^N(t)\geq (\lambda t + \varepsilon)N}\to 0\qquad\mbox{as}\qquad N\to\infty.$$
Define $M(t,\qq^\infty):=\min\big\{k\geq 1: \sum_{i=1}^{k-1}(1-q_i^\infty)>\lambda t\big\},$
and choose 
$$\varepsilon = \sum_{i=1}^{M(t,\qq^\infty)-1}(1-q_i^\infty)-\lambda t >0.$$
Note that since $\qq^\infty\in \mathcal{S}\subset \ell_1$, $M(t,\qq^\infty)$ exists and is finite for all $t\geq 0.$
We now claim that the probability that in the interval $[0,t]$ a task is assigned to some server with $M(t,\qq^\infty)$ active tasks tends to 0, as $N\to\infty.$
Indeed, in order for a task to be assigned to some server with $M(t,\qq^\infty)-1$ active tasks, all the servers must have at least $M(t,\qq^\infty)-1$ active tasks. 
Now, the minimum number of tasks required for this, is given by 
$\sum_{i=1}^{M(t,\qq^\infty)-1}(N-Q_i^N(0))$.
Therefore, the proof is complete by observing that 
\begin{align*}
\mathbbm{P}\Big(A^N(t)\geq \sum_{i=1}^{M(t,\qq^\infty)-1}(N-Q_i^N(0))\Big) = \Pro{A^N(t)\geq \Big(\lambda t +\frac{\varepsilon}{2}\Big) N}\to 0, \quad \mbox{as}\quad N\to\infty.
\end{align*}
\end{proof}

\begin{proof}[Proof of Proposition~\ref{prop:rel compactness}]
The proof goes in two steps. We first prove the relative compactness, and then show that the limit satisfies~\eqref{eq:rel compact}.

Observe from \cite[Proposition 3.2.4]{EK2009} that, to prove the relative compactness of the sequence of processes $\big\{(\mathbf{q}^N(\cdot),\alpha^N)\big\}_{N\geq 1}$, it is enough to prove relative compactness of the individual components.
Note that, from Prohorov's theorem~\cite[Theorem 3.2.2]{EK2009}, $\mathfrak{L}$ is compact, since $G$ is compact. Now, relative compactness of $\big\{\alpha^N\big\}_{N\geq 1}$ follows from the compactness of $\mathfrak{L}$ under the topology of weak convergence of measures and Prohorov's theorem.
To claim the relative compactness of $\big\{\mathbf{q}^N(\cdot)\big\}_{N\geq 1}$, we will verify the conditions of Theorem~\ref{th:from EK}. 

Observe that Theorem~\ref{th:from EK} (a) requires to show tightness of the sequence $\big\{\mathbf{q}^N(t)\big\}_{N\geq 1}$ for each fixed (rational) $t\geq 0$.
Fix any $t\geq 0.$
Due to Lemma~\ref{lem:trivialbound}, we know  
$$\lim_{N\to\infty}\Pro{q_i^N(t)\leq q_i^N(0),\quad \forall\ i\geq M(t,\qq^\infty)}= 1.$$
Also, $\qq^N(0)\dto \qq^\infty$ with respect to $\ell_1$ topology. In particular, $\big\{\qq^N(0)\big\}_{N\geq 1}$ is tight in $\ell_1$.
Therefore, using (ii)$\implies$(i) in Lemma~\ref{lem:tightcond} we obtain, for any $\varepsilon>0$,
\begin{align*}
\lim_{k\to\infty}\varlimsup_{N\to\infty} \mathbbm{P}\Big(\sum_{i\geq k}q_i^N(t)>\varepsilon\Big)
&\leq \lim_{k\to\infty}\varlimsup_{N\to\infty} \mathbbm{P}\Big(\sum_{i\geq k}q_i^N(0)>\varepsilon\Big)=0.
\end{align*}
Also, since $\qq^N(t)\in \mathcal{S}\subseteq [0,1]^b$, which is compact with respect to the product topology,  $\big\{\qq^N(t)\big\}_{N\geq 1}$ is tight with respect to the product topology.
Hence using (i)$\implies$(ii) in Lemma~\ref{lem:tightcond} we conclude that the sequence $\big\{\mathbf{q}^N(t)\big\}_{N\geq 1}$ is tight in $\ell_1.$
For condition (b),  first note that 
for all $i = 1,\ldots, b$.
\begin{align*}
&|q_i^N(t_1)-q_i^N(t_2)|
\leq \lambda \alpha^N([t_1,t_2]\times\mathcal{R}_i)+\int_{t_1}^{t_2} (q^N_i(s)-q^N_{i+1}(s))\dif s \\
&\hspace{4cm}+\frac{1}{N}\big|M_{A,i}^N(t_1)-M_{D,i}^N(t_1)-M_{A,i}^N(t_2)+M_{D,i}^N(t_2)\big| + o(1).
\end{align*}
Thus, 
\begin{equation}\label{mart-norm-ub}
\begin{split}
&\norm{\qq^N(t_1)-\qq^N(t_2)}\\
&\leq \lambda \sum_{i=1}^b\alpha^N([t_1,t_2]\times\mathcal{R}_i)+\int_{t_1}^{t_2} \sum_{i=1}^b (q^N_i(s)-q^N_{i+1}(s))\dif s \\
&\hspace{3cm}+\frac{1}{N}\sum_{i=1}^b\big|M_{A,i}^N(t_1)-M_{D,i}^N(t_1)-M_{A,i}^N(t_2)+M_{D,i}^N(t_2)\big|+o(1)\\
&\leq \lambda (t_1-t_2)+\int_{t_1}^{t_2} q^N_1(s)\dif s +\frac{1}{N}\sum_{i=1}^b\big|M_{A,i}^N(t_1)-M_{D,i}^N(t_1)-M_{A,i}^N(t_2)+M_{D,i}^N(t_2)\big|+o(1)\\
&\leq (\lambda+1) (t_1-t_2)+\frac{1}{N}\sum_{i=1}^b\big|M_{A,i}^N(t_1)-M_{D,i}^N(t_1)-M_{A,i}^N(t_2)+M_{D,i}^N(t_2)\big|+o(1).
\end{split}
\end{equation}
Now, from the $\ell_1$ convergence of scaled martingales in Proposition~\ref{prop:mart zero1}, we get, for any $T\geq 0$,
$$\sup_{t\in[0,T]}\frac{1}{N}\sum_{i=1}^b|M_{A,i}^N(t_1)-M_{D,i}^N(t_1)-M_{A,i}^N(t_2)+M_{D,i}^N(t_2)|\pto 0.$$
Observe that the proof of the relative compactness of $\big\{\qq^N(t)\big\}_{t\geq 0}$ is complete if we show that for any $\eta>0$, there exists a $\delta >0$ and a finite partition $(t_j)_{i=1}^n$ of $[0,T]$ with $\min_j|t_{j}-t_{j-1}|>\delta$ such that 
\begin{equation}
\varlimsup_{N\to\infty}\mathbbm{P}\Big(\max_j \sup_{s,t\in [t_{j-1},t_j)}\norm{\qq^N(s)-\qq(t)} \geq \eta \Big) < \eta.
\end{equation}
Now, \eqref{mart-norm-ub} implies that, for any finite partition $(t_j)_{j= 1}^n$ of $[0,T]$,
\begin{align*}
\max_j \sup_{s,t\in [t_{j-1},t_j)} \norm{\qq^N(s)-\qq^N(t)} &\leq (\lambda+1) \max_j (t_{j}-t_{j-1})+\zeta_N,
\end{align*}
where $\Pro{\zeta_N>\eta/2}<\eta$ for all sufficiently large $N$. Now take $\delta = \eta/(4(\lambda+1))$ and any partition with $\max_j(t_j-t_{j-1})< \eta/(2(\lambda+1))$ and $\min_j(t_j-t_{j-1})>\delta$. 
Now on the event $\big\{\zeta_N\leq \eta/2\big\}$,  
$$\max_i \sup_{s,t\in [t_{i-1},t_i)}\norm{\qq^N(s)-\qq^N(t)} \leq \eta.$$
Therefore, for all sufficiently large $N$,
\begin{align*}
&\mathbbm{P}\Big(\max_j \sup_{s,t\in [t_{j-1},t_j)}\norm{\qq^N(s)-\qq^N(t)} \geq \eta \Big)
\leq \Pro{\zeta_N>\eta/2}\leq \eta.
\end{align*}

To prove that the limit $(\qq(\cdot),\alpha)$ of any convergent subsequence satisfies~\eqref{eq:rel compact}, we will use the continuous-mapping theorem~\cite[Theorem~3.4.1]{W02}.
Specifically, we will show that the right side of~\eqref{eq:martingale rep assumption 2-2} is a continuous map of suitable arguments.
Let $\big\{\qq(t)\big\}_{t\geq 0}$ and $\big\{\yy(t)\big\}_{t\geq 0}$ be an $\mathcal{S}$-valued and an $\ell_1$-valued c\'adl\'ag function, respectively. 
Also, let $\alpha$ be a measure on the measurable space $([0,\infty)\times G, \mathcal{C}\otimes\mathcal{G})$. Then for $\qq^0\in \mathcal{S}$, define for $i\geq 1$,
$$F_i(\qq,\alpha,\qq^0,\yy)(t):=q_i^0+y_i(t)+\lambda \alpha([0,t]\times\mathcal{R}_i)-\int_0^t(q_i(s) - q_{i+1}(s))\dif s.$$
Observe that it is enough to show $\FF=(F_1,\ldots,F_b)$ is a continuous
operator. 
Indeed, in that case the right side of~\eqref{eq:martingale rep assumption 2-2} can be written as $\FF(\qq^N,\alpha^N,\qq^N(0),\yy^N)$, where $\yy^N=(y_1^N,\ldots,y_b^N)$ with $y_i^N= (M_{A,i}^N-M_{D,i}^N)/N$, and since each argument converges, we will get the convergence to the right side of~\eqref{eq:rel compact}.
Therefore, we now prove the continuity of $\FF$ below. 
In particular assume that (a)~the sequence of processes $\big\{(\qq^N,\yy^N)\big\}_{N\geq 1}$ converges to $(\qq,\yy)$ with respect to $\ell_1$ topology, (b)~for any fixed $t\geq 0$, the sequence $\big\{\big(\alpha^N([0,t],\mathcal{R}_i)\big)_{i\geq 1}\big\}_{N\geq 1}$ in $\ell_1$ converges to $\big(\alpha([0,t],\mathcal{R}_i)\big)_{i\geq 1}$, and (c)~the sequence of $\mathcal{S}$-valued random  variables $\qq^N(0)$
 converges to $\qq(0)$ with respect to $\ell_1$ topology.
 
 Fix any $T\geq 0$ and $\varepsilon>0$.
 \begin{enumerate}[{\normalfont (i)}]
 \item  Due to (a) above, choose $N_1\in\N$, such that $\sup_{t\in[0,T]}\norm{\qq^N(t)-\qq(t)}<\varepsilon/(4T)$. In that case, observe that
 \begin{align*}
 \sup_{t\in [0,T]}\int_0^t|q_1^N(t) - q_1(t)|\dif s& \leq T\sup_{t\in [0,T]}\norm{\qq^N(t))-\qq(t))} <\frac{\varepsilon}{4}.
 \end{align*}
 \item Again, due to (a), choose $N_2\in\N$, such that $\sup_{t\in[0,T]}\norm{\yy^N(t)-\yy(t)}<\varepsilon/4$.
 \item   We now claim that for the $\epsilon > 0$ given above there is an $N_3 \in {\mathbb N}$
such that
 \begin{equation}\label{eq:measureconv}
 \lambda\sum_{i\geq 1}  \left|\alpha^N([0,T]\times\mathcal{R}_i)-\alpha([0,T]\times\mathcal{R}_i)\right|<\frac{\varepsilon}{4}.
 \end{equation}
 Observe that we only know the weak convergence of the sequence of measures~$\alpha^N$, and therefore we cannot directly make assumption (b) above.
We are therefore about to  show that assumption (b) is valid in our case and that it follows from weak convergence.
Indeed, since $\qq^\infty\in \mathcal{S}\subseteq \ell_1$, there exists $\hat{M}(\qq^\infty)$, such that 
$q_{\hat{M}(\qq^\infty)}^\infty<1$, and consequently $q_i^\infty<1$ for all $i\geq \hat{M}(\qq^\infty)$.
Also, due to Lemma~\ref{lem:trivialbound}, 
$$\lim_{N\to\infty}\mathbbm{P}\Big(\sup_{t\in[0,T]}q_i^N(t)\leq q_i^N(0)\quad\mbox{for all}\quad i\geq M(T,\qq^\infty)\Big)= 1.$$
Thus, if $N_0 := \max\big\{\hat{M}(\qq^\infty), M(T,\qq^\infty)\big\}$, then 
$$\lim_{N\to\infty}\mathbbm{P}\Big(\sup_{t\in[0,T]}q_i^N(t)<1\quad\mbox{for all}\quad i\geq N_0\Big)= 1.$$
This implies
$$\sum_{i\geq N_0}\alpha^N([0,T]\times\mathcal{R}_i)\pto \sum_{i\geq N_0}\alpha([0,T]\times\mathcal{R}_i)=0.$$
Also, due to weak convergence of $\alpha^N$, 
$$\sum_{i< N_0}\alpha^N([0,T]\times\mathcal{R}_i)\pto \sum_{i< N_0}\alpha([0,T]\times\mathcal{R}_i).$$ 
 \item  Finally, due to (c), choose $N_4\in\N$, such that $\norm{\qq^N(0)-\qq(0)}<\varepsilon/4$.
 \end{enumerate}
Let $\hat{N}=\max\big\{N_1,N_2,N_3,N_4\big\}$, then for $N\geq \hat{N}$,
\begin{align*}
&\sup_{t\in [0,T]}\norm{\FF(\qq^N,\alpha^N,\qq^N(0),\yy^N)-\FF(\qq,\alpha,\qq(0),\yy)}(t)<\varepsilon.
\end{align*}
Thus the proof of continuity of $\FF$ is complete.
\end{proof}

To characterize the limit in~\eqref{eq:rel compact}, for any $\qq\in \mathcal{S}$, define the Markov process  $\ZZ_{\qq}$ on $G$ as
\begin{equation}\label{eq:slowprocess}
\ZZ_{\qq} \rightarrow 
\begin{cases}
\ZZ_{\qq}+e_i& \quad\mbox{ at rate }\quad q_i-q_{i+1},\\
\ZZ_{\qq}-e_i& \quad\mbox{ at rate }\quad \lambda \ind{\ZZ_\qq\in\mathcal{R}_{i}},
\end{cases}
\end{equation}
where $e_i$ is the $i^{\mathrm{th}}$ unit vector, $i=1,\ldots,b$.

\begin{proof}[{Proof of Theorem~\ref{th:genfluid}}]
Having proved the relative compactness in Proposition~\ref{prop:rel compactness},  it follows from analogous arguments as used in the proofs of~\cite[Lemma 2]{HK94} and \cite[Theorem 3]{HK94}, that the limit of any convergent subsequence of the sequence of processes $\big\{\qq^N(t)\big\}_{t\geq 0}$ satisfies
\begin{equation}
q_i(t) = q_i(0)+\lambda \int_0^t \pi_{\qq(s)}(\mathcal{R}_i)\dif s - \int_0^t (q_i(s)-q_{i+1}(s))\dif s, \quad i=1,2,\ldots,b,
\end{equation}
for \emph{some} stationary measure $\pi_{\qq(t)}$ of the Markov process  $\ZZ_{\qq(t)}$ described in~\eqref{eq:slowprocess} satisfying $\pi_{\qq}\big\{\ZZ: Z_i=\infty\big\}=1$ if $q_i<1$. 

Now it remains to show that $\qq(t)$ \emph{uniquely} determines $\pi_{\qq(t)}$, and that $\pi_{\qq(s)}(\mathcal{R}_i)=p_{i-1}(\qq(s))$ described in~\eqref{eq:fluid-gen}. 
As mentioned earlier, in this proof we will now assume the specific assignment probabilities in~\eqref{eq:partition}, corresponding to the ordinary JSQ policy.
To see this, fix any $\qq=(q_1,\ldots,q_b)\in \mathcal{S}$.  
Observe that due to summability of the components of~$\qq$, there exists $0\leq m<\infty$, such that $q_{m+1}<1$ and $q_1=\ldots=q_m=1$,
with the convention that $q_0\equiv 1$ and $q_{b+1}\equiv 0$ if $b<\infty$. In that case,
$$\pi_{\qq}\big(\big\{Z_{m+1}=\infty, Z_{m+2}=\infty,\ldots,Z_b=\infty\big\}\big)=1.$$
Also, 
note that $q_i = 1$ forces $\dif q_i/\dif t \leq 0$, i.e., $\lambda \pi_{\qq}(\mathcal{R}_i) \leq q_i-q_{i+1}$ for all $i = 1, \ldots, m$, and in particular $\pi_{\qq}(\mathcal{R}_i) = 0$ for all $i = 1, \ldots, m - 1.$ Thus,
$$\pi_{\qq}\big(\big\{Z_1=0,Z_2=0,\ldots,Z_{m-1}=0\big\}\big)=1.$$

Therefore, $\pi_\qq$ is determined only by the stationary distribution of the $m^{\mathrm{th}}$ component, which can be described as a birth-death process
\begin{equation}\label{eq:bdprocess}
Z \rightarrow 
\begin{cases}
Z+1& \quad\mbox{ at rate }\quad q_m-q_{m+1}\\
Z-1& \quad\mbox{ at rate }\quad \lambda\ind{Z>0}
\end{cases}
\end{equation}
and let $\pi^{(m)}$ be its stationary distribution. 
Now it is enough to show that $\pi^{(m)}$ is uniquely determined by $\qq$. 
First observe that the process on $\bZ$ described in~\eqref{eq:bdprocess} is reducible, and can be decomposed into
two irreducible classes given by $\mathbbm{Z}$ and $\{\infty\}$, respectively.
Therefore, if $\pi^{(m)}(Z=\infty)=0$ or $1$, then it is unique. 
Indeed, if $\pi^{(m)}(Z=\infty)=0$, then $Z$ is birth-death process on $\mathbbm{Z}$ only, and hence it has a unique stationary distribution. 
Otherwise, if $\pi^{(m)}(Z=\infty)=1$, then it is trivially unique. 
Now we distinguish between two cases depending upon whether $q_m-q_{m+1}\geq \lambda$ or not.

Note that if $q_m-q_{m+1}\geq\lambda$, then $\pi^{(m)}(Z\geq k)=1$ for all $k\geq 0$. 
On $\bZ$ this shows that $\pi^{(m)}(Z=\infty)=1$.
Furthermore, if $q_m-q_{m+1}<\lambda$, we will show that $\pi^{(m)}(Z=\infty)=0$.
On the contrary, assume $\pi^{(m)}(Z=\infty)=\varepsilon\in (0,1]$.
Also, let $\hat{\pi}^{(m)}$ be the unique stationary distribution of the birth-death process in~\eqref{eq:bdprocess} on $\mathbbm{Z}$.
Therefore, 
$$\pi_\qq(\mathcal{R}_m)=\pi^{(m)}(Z>0)=(1-\varepsilon)\hat{\pi}^{(m)}(Z>0)+\varepsilon = (1-\varepsilon)\frac{q_m-q_{m+1}}{\lambda}+\varepsilon.$$
Substituting into the differential form of the fluid equation~\eqref{eq:fluidfinal} at the given time $t$, we obtain that 
\begin{align*}
\frac{\dif q_m(t)}{\dif t} = \lambda \Big[(1-\varepsilon)\frac{q_m-q_{m+1}}{\lambda}+\varepsilon \Big] - (q_m - q_{m+1}) = -\varepsilon(q_m-q_{m+1}) +\lambda \varepsilon >0,
\end{align*}
where the last inequality follows since we are considering the case when $q_m - q_{m+1}<\lambda$.
Now since $q_m(t)=1$, this leads to a contradiction for any $\varepsilon>0$, and hence it must be the case that $\pi^{(m)}(Z=\infty)=0$. 

Therefore, for all $\qq\in \mathcal{S}$, $\pi_\qq$ is uniquely determined by $\qq$. 
Furthermore, we can identify the expression for $\pi_q(\mathcal{R}_i)$ as
\begin{equation}
\pi_\qq(\mathcal{R}_i)=
\begin{cases}
\min\big\{(q_m-q_{m+1})/\lambda,1\big\}& \quad\mbox{ for }\quad i=m,\\
1- \min\big\{(q_m-q_{m+1})/\lambda,1\big\} & \quad\mbox{ for }\quad i=m+1,\\
0&\quad \mbox{ otherwise,}
\end{cases}
\end{equation}
and hence $\pi_{\qq(s)}(\mathcal{R}_i)=p_{i-1}(\qq(s))$ as claimed.
\end{proof}

\subsection{Equivalence on fluid scale}\label{ssec:equivjsq}
Having proved Theorem~\ref{th:genfluid}, it suffices to prove the universality property stated in the next proposition. This will complete the proof of Theorem~\ref{fluidjsqd}.
\begin{proposition}\label{prop:samefluid}
If $d(N)\to\infty$ as $N\to\infty$, then the JSQ$(d(N))$ scheme and the ordinary JSQ policy have the same fluid limit.
\end{proposition}
The proof of the above proposition uses the S-coupling 
 results from Section~\ref{sec:coupling}, and consists of three steps:
\begin{enumerate}[{\normalfont (i)}]
\item First we show that if $n(N)/N\to 0$ as $N\to\infty$, then the MJSQ$(n(N))$ scheme has the same fluid limit as the ordinary JSQ policy.
\item Then we apply Corollary~\ref{cor:bound} to prove that as long as $n(N)/N\to 0$, \emph{any} scheme from the class CJSQ$(n(N))$ has the same fluid limit as the ordinary JSQ policy.
\item Next, using Propositions~\ref{prop:stoch-ord2} and~\ref{prop:differ} we establish that if $d(N)\to\infty$, then for \emph{some} $n(N)$ with $n(N)/N\to 0$, the  JSQ$(d(N))$ scheme and the JSQ$(n(N),d(N))$ scheme have the same fluid limit. The proposition then follows by observing that the JSQ$(n(N),d(N))$ scheme belongs to the class CJSQ$(n(N))$.
\end{enumerate}

\begin{proof}[Proof of Proposition~\ref{prop:samefluid}]
First, to show Claim~(i) above, define $\bar{N}=N-n(N)$ and $\bar{\lambda}(\bar{N})=\lambda(N)$.
Observe that the MJSQ$(n(N))$ scheme with $N$ servers can be thought of as the ordinary JSQ policy with $\bar{N}$ servers and arrival rate $\bar{\lambda}(\bar{N})$.
Also, since $n(N)/N\to 0$,
\begin{align*}
\frac{\bar{\lambda}(\bar{N})}{\bar{N}}=\frac{\lambda(N)}{N-n(N)}\to \lambda\quad \text{as}\quad \bar{N}\to\infty.
\end{align*}
Furthermore, observe that the fluid limit of the JSQ policy in Theorem~\ref{th:genfluid} as given by~\eqref{eq:fluidfinal} is characterized by the parameter $\lambda$ only, and hence the fluid limit of the MJSQ$(n(N))$ scheme is the same as that of the ordinary JSQ policy.

Second, observe from the fluid limit of the JSQ policy that if $\lambda< 1$, then for any buffer capacity $b\geq 1$, and any starting state, the fluid-scaled cumulative overflow is negligible, i.e., for any $t\geq 0$, $L^N(t)/N\pto 0$.
Since the above fact is induced by the fluid limit only, the same holds for the MJSQ$(n(N))$ scheme.
Therefore, using the lower and upper bounds in Corollary~\ref{cor:bound} and the tail bound in Proposition~\ref{prop:stoch-ord}, we obtain Claim~(ii) above.

Finally, choose $n(N)=N/\sqrt{d(N)}$, and consider the JSQ$(n(N),d(N))$ scheme. 
Since $d(N)\to\infty$, it is clear that $n(N)/N\to 0$ as $N\to\infty$.
Also, if $\Delta^N(T)$ denotes the cumulative number of times that the JSQ($d(N)$) scheme and JSQ$(n(N),d(N))$ scheme differ in decision up to time $T$, then Proposition~\ref{prop:differ} yields
\begin{align*}
\Pro{\Delta^N(T)\geq \varepsilon N\given A^N(T)}\leq \frac{A^N(T)}{\varepsilon N}\left(1-\frac{n(N)}{N}\right)^{d(N)}=\frac{A^N(T)}{\varepsilon N}\left(1-\frac{1}{\sqrt{d(N)}}\right)^{d(N)}.
\end{align*}
Since $\big\{A^N(T)/N\big\}_{N\geq 1}$ is a tight sequence of random variables, we have
\begin{align*}
\frac{A^N(T)}{\varepsilon N}\left(1-\frac{1}{\sqrt{d(N)}}\right)^{d(N)}\pto 0
\quad\text{as}\quad N\to\infty,
\end{align*}
and hence, $\Delta^N(T)/N\pto 0$. 
Therefore, applying the $\ell_1$ distance bound stated in Proposition~\ref{prop:stoch-ord2}, we obtain Claim~(iii).
The proof is then completed by observing that the JSQ$(n(N),d(N))$ scheme belongs to the class CJSQ$(n(N))$.
\end{proof}

\begin{proof}[Proof of Theorem~\ref{th:batch}]
For any $\varepsilon>0$, define 
$$T^N_\varepsilon:=\inf\big\{t:Q_1^{\sss d(N)}(t)>(\lambda+\varepsilon)N\big\}.$$
Now the proof consists of two main steps. First we show that if $d(N)\geq \ell(N)/(1-\lambda-\varepsilon)$ for some $\varepsilon>0$, then there exists an $\varepsilon'>0$, such that if for some $T>0$, $\Pro{T^N_{\varepsilon'}>T}\to 1$ as $N\to\infty$,
 then the number of times that the JSQ($d(N)$) scheme and the ordinary JSQ policy differ in decision in $[0,T]$ is $\op(N)$. This then implies that up to such a time $T$, it is enough to consider the fluid limit of the ordinary JSQ policy with batch arrivals. Second, we show that if the conditions stated in Theorem~\ref{th:batch} hold, then for any finite time $T> 0$, $\Pro{T^N_{\varepsilon'}>T}\to 1$ as $N\to\infty$. This will complete the proof.

To prove the first part, consider the JSQ($d(N)$) scheme in case of batch arrivals. Choose $\varepsilon'=\varepsilon/2$, and assume that $T>0$ is such that $\Pro{T^N_{\varepsilon'}>T}\to 1$ as $N\to\infty$.
Let $I_i$ denote the number of idle servers among $d(N)$ randomly chosen servers for the $i^{\mathrm{th}}$ batch arrival, and define $W^N(t)$ to be the cumulative number of tasks that have not been assigned to some idle server, up to time $t$. If $A^N(t)$ denotes the number of batch arrivals that occurred up to time $t$, then 
$$W^N(t)=\sum_{i=1}^{A^N(t)}[\ell(N)-I_i]^+.$$
We show that $W^N(t)/N\pto 0$ for all $t\leq T$ for $d(N)=\ell(N)/(1-\lambda-\varepsilon)$. Observe that $I_i$ follows a Hypergeometric distribution with sample size $d(N)$, and population size $N$ containing  $N-Q_1^N(t)\geq (1-\lambda-\varepsilon/2) N$ successes. Define $J_i$ to be distributed as $d(N)-I_i$. Then
$$[\ell(N)-I_i]^+=k\iff J_i=d(N)-\ell(N)+k.$$
Therefore, for $c=1-\lambda-\varepsilon/2$ we have,
$$\E{[\ell(N)-I_i]^+}=\sum_{k\geq 1}k\Pro{J_i=(1-c)d(N)+k}\leq d(N)\Pro{J_i\geq (1-c)d(N)}.$$
Now, from \cite{H63, LP14}, we know
$$\Pro{J_i\geq (1-c)d(N)}\leq \exp(-d(N)H(\lambda,c)),$$
where $$H(\lambda,c)=(1-c)\log\left(\frac{1-c}{\lambda}\right)+c\log\left(\frac{c}{1-\lambda}\right)>0,$$
since $c<1-\lambda.$
Therefore, 
\begin{equation}
\begin{split}
\Pro{W^N(t)>\varepsilon N}
&\leq\frac{\E{W^N(t)}}{\varepsilon N}\\
&\leq \frac{d(N)}{\varepsilon N}\times  \frac{\lambda(N) t}{\ell(N)}\times \exp(-d(N)H(\lambda,c))\\
&=O(\exp(-d(N)H(\lambda,c))).
\end{split}
\end{equation}
This implies that whenever $\ell(N)\to\infty$, if $d(N)=\ell(N)/(1-\lambda-\varepsilon/2)$, then $W^N(t)$ is $o_P(N)$ for all $t\leq T$. 
Now the analysis of the batch arrival with ordinary JSQ policy in Theorem~\ref{th:batchjsq} below, up to time $T$, shows that the process
$\big\{\qq^N(t)\big\}_{0\leq t\leq T}$ converges to the deterministic limit $\big\{\qq(t)\big\}_{0\leq t\leq T}$, described by \eqref{eq:batch}.

Therefore, it is enough to show that any $T>0$ satisfies the required criterion. This can be seen by observing that for any $T\geq0$, and any $\varepsilon'>0$,
\begin{align*}
&\Pro{T_{\varepsilon'}^N\leq T}\leq \Pro{T_{\varepsilon'/2}^N< T}\\
&\leq\Pro{\sup_{t\in[0,T]}Q_1^{\sss d(N)}(t)>(\lambda+\varepsilon'/2)N}\\
&\leq\Pro{\sup_{t\in[0,T]}Q_1^\JSQ(t)>(\lambda+\varepsilon'/4)N}\times\Pro{\sup_{t\in[0,T]}|Q_1^\JSQ(t)-Q_1^{\sss d(N)}(t)|\leq \frac{N\varepsilon'}{4}}\\
&\hspace{2.5cm}+\Pro{\sup_{t\in[0,T]}|Q_1^\JSQ(t)-Q_1^{\sss d(N)}(t)|> \frac{N\varepsilon'}{4}}\longrightarrow 0\quad\mathrm{as}\quad N\to\infty.
\end{align*}
Therefore the proof is complete.
\end{proof}

\begin{theorem}{\normalfont (Batch arrivals JSQ)}
\label{th:batchjsq}
Consider the batch arrival scenario with growing batch size $\ell(N)\to\infty$ and $\lambda(N)/N\to\lambda<1$ as $N\to\infty$. For the JSQ policy, if $q^{\sss d(N)}_1(0)\pto q_1^\infty\leq \lambda$, and $q_i^{\sss d(N)}(0)\pto 0$ for all $i\geq 2$, then the sequence of processes
$\big\{\qq^{\sss d(N)}(t)\big\}_{t\geq 0}$ converges weakly to the limit $\big\{\qq(t)\big\}_{t\geq 0}$, described as follows: 
\begin{equation}\label{eq:batchjsq}
q_1(t) = \lambda + (q_1^\infty-\lambda)\e^{-t},\quad
q_i(t)\equiv 0\quad \mathrm{for\ all}\quad i= 2,\ldots,b.
\end{equation}
\end{theorem}
\begin{proof}
Fix any finite time $T\geq 0$.
To analyze the JSQ policy with batch arrivals, observe that before time $T$, all the arriving tasks join idle servers. Therefore, assuming $Q_2^N(0)=0$, for all $t\leq T$, the evolution for $Q_1^N$ can be written as
\begin{equation}
Q_1^N(t)=Q_1^N(0)+\ell(N)A\left(t\lambda(N)/\ell(N)\right)-D\left(\int_0^tQ_1^N(s)ds\right),
\end{equation}
where $A$ and $D$ are independent unit-rate Poisson processes.
Using the random time change of unit-rate Poisson processes \cite[Lemma~3.2]{PTRW07}, and applying the arguments in \cite[Lemma~3.4]{PTRW07}, the above process scaled by $N$, then admits the martingale decomposition
\begin{equation}\label{eq:fluid-batch}
q_1^N(t)=q_1^N(0)+\frac{M^N_1(t)}{N}+\lambda t-\frac{M^N_2(t)}{N}-\int_0^tq_1^N(s)ds,
\end{equation}
where 
\begin{align*}
M^N_1(t)&=\ell(N)A\left(t\lambda(N)/\ell(N)\right)-t\lambda(N),\\
M^N_2(t)&=D\left(\int_0^tQ_1^N(s)ds\right)-\int_0^tQ_1^N(s)ds,
\end{align*}
are square integrable martingales with respective quadratic variation processes given by
\begin{align*}
\langle M^N_1\rangle(t)&=t\lambda(N),\\
\langle M^N_2\rangle(t)&=\int_0^tQ_1^N(s)ds.
\end{align*}
Now, since for any $T\geq 0$, $\langle M^N_1\rangle(T)/N^2\to 0$, and $\langle M^N_2\rangle(T)/N^2\pto 0$, from the stochastic boundedness criterion for square integrable martingales \cite[Lemma~5.8]{PTRW07}, we get that both $\big\{M_1^N(t)/N\big\}_{t\geq 0}\dto 0$ and $\big\{M_2^N(t)/N\big\}_{t\geq 0}\dto 0$. Therefore, from the continuous mapping theorem and \eqref{eq:fluid-batch}, it follows that $\big\{q_1^N(t)\big\}_{t\geq 0}$ as $N\to\infty$ converges weakly to a deterministic limit described by the integral equation
\begin{equation}
q_1(t)=q_1^\infty+\lambda t-\int_0^tq_1(s)ds
\end{equation}
having~\eqref{eq:batch} as the unique solution.
This completes the proof of the fluid limit of JSQ with batch arrivals.
\end{proof}

\subsection{Global stability and interchange of limits}\label{ssec:globstab}

To prove the interchange of limits result stated in Proposition~\ref{th:interchange}, we will establish the global stability of the fixed point, i.e., all fluid paths converge to the fixed point in~\eqref{eq:fixed point} as $t\to\infty$. 
This is formally stated in the following lemma.
\begin{lemma}\label{lem:global-stab}
Let $\qq(t)$ be the fluid limit, i.e., the solution of the dynamical system described by the system of integral equations in~\eqref{eq:fluid}.
For any $\qq^\infty\in \mathcal{S}$, if $\qq(0) = \qq^\infty$, then $\qq(t)\to
\qq^*$ as $t\to\infty$, where $\qq^*$ is defined as in \eqref{eq:fixed point}.
\end{lemma}
In case of the JSQ$(d)$ scheme with fixed $d$, the global stability is proved by constructing a Lyapunov function that measures the `distance' (in terms of a weighted $L_1$-norm) between the trajectory and the fixed point, and that strictly decreases everywhere except at the fixed point, see~\cite[Theorem 3.6]{Mitzenmacher1996}.
In case of the ordinary JSQ policy however, we can exploit a more direct method to establish the global stability, as further detailed below.
\begin{proof}[Proof of Lemma~\ref{lem:global-stab}]
The proof follows in two steps: we will first establish that as $t\to\infty$, $q_1(t)\to\lambda<1$, and then show that $q_2(t)\to 0$.

Observe that the rate of change of $q_1(t)$ is $\lambda p_0(\qq(t))-(q_1(t)-q_2(t))$.
For any $\varepsilon\geq 0$, if $q_1(t)\leq \lambda-\varepsilon$, then $p_0(\qq(t))=1$, so that the rate of change is  $\lambda -(q_1(t)-q_2(t))\geq \varepsilon$, i.e., positive and bounded away from zero when $\varepsilon>0$.
Also, $q_1(t)$ cannot decrease if $q_1(t)\leq\lambda$.
This shows that for all $\varepsilon>0$, there exists a time $t_0 = t_0(\varepsilon, \qq^\infty)$, such that, $q_1(t)\geq \lambda -\varepsilon$ for all $t\geq t_0$.
Thus, $\liminf_{t\to\infty} q_1(t)\geq\lambda$.

On the other hand, we claim that $\limsup_{t\to\infty} q_1(t)\leq \lambda$.
Suppose not, i.e., assume $\limsup_{t\to\infty}q_1(t) = \lambda+\varepsilon$ for some $\varepsilon>0$.
Because $q_1(t)$ is non-decreasing when $q_1(t)\leq \lambda$, there must exist a $t_0$ such that $q_1(t)\geq \lambda$ $\forall\ t\geq t_0$.
The high-level idea behind the claim is as follows.
If $q_1(t)$ were to remain above $\lambda$ by a non-vanishing margin, then the cumulative number of departures would exceed the cumulative number of arrivals by an infinite amount, which cannot occur since the initial number of tasks is bounded.
More formally,
\begin{align*}
\sum_{i=1}^b q_i(t) &= \sum_{i=1}^bq_i(t_0) + \lambda\int_{t_0}^t\sum_{i=1}^b p_{i-1}(\qq(s))\dif s - \int_{t_0}^t q_1(s)\dif s\\
& \leq \sum_{i=1}^bq_i(t_0)  - \int_{t_0}^t [q_1(s) - \lambda]^+\dif s,
\end{align*}
and thus,
$$\int_{t_0}^t [q_1(s) - \lambda]^+\dif s\leq \sum_{i=1}^b q_i(t) - \sum_{i=1}^bq_i(t_0)<\infty.$$
This provides a contradiction with $\limsup_{t\to\infty} q_1(t) = \lambda +\varepsilon$, since the rate of decrease of $q_1(t)$ is at most~1.
Therefore, $q_1(t)\to \lambda$ as $t\to\infty$.  

Consequently, for any $\qq^\infty\in\mathcal{S}$ and $\varepsilon>0$, if $\qq(0) = \qq^\infty$, then there exists a time $t_2 = t_2(\qq^\infty, \varepsilon)<\infty$, such that $q_1(t)\leq \lambda+\varepsilon$ for all $t\geq t_2$.
Thus choosing $\varepsilon = (1-\lambda)/2$ say, for all $t\geq t_2$, $q_1(t)<1$, and thus $p_0(\qq(t))=1$, i.e., $\sum_{i=2}^b p_{i-1}(\qq(t))=0$.
Define $q_{2+}(t) := \sum_{i = 2}^b q_i(t)$.
Observe that
\begin{align*}
q_{2+}(t)&= q_{2+}(t_2)+\lambda\int_{t_2}^t\sum_{i=2}^b p_{i-1}(\qq(s))\dif s- \int_{t_2}^tq_2(s)\dif s\\
&=  q_{2+}(t_2)- \int_{t_2}^tq_2(s)\dif s\qquad \mbox{for all}\quad t\geq t_2,
\end{align*}
which implies $q_2(t)\leq q_{2+}(t_2)\e^{-(t-t_2)}$.
Thus, $q_2(t)$ and consequently, $q_{2+}(t)$ converges to 0 as $t\to\infty$.
This completes the proof of global stability of the fixed point.
\end{proof}

\begin{proof}[Proof of Proposition~\ref{th:interchange}]
The proof follows in two steps: (i) we first establish that the sequence of stationary measures $\big\{\pi^{\sss d(N)}\big\}_{N\geq 1}$ is tight, and then (ii) show the interchange of limits.

(i) Observe that if $b<\infty,$ then the space $[0,1]^b$ is compact, and hence Prohorov's theorem implies that $\big\{\pi^{\sss d(N)}\big\}_{N\geq 1}$ is tight. 
Now assume $b=\infty.$
For any two positive integers $d_1\leq d_2$, note that at each arrival, the JSQ$(d_2)$ scheme polls more servers than the JSQ$(d_1)$ scheme.
Thus using the S-coupling and Proposition~\ref{prop:det ord}, we can conclude for every~$N$,
$$\sum_{i\geq m} Q_i^{d_2}\leq_{st}\sum_{i\geq m} Q_i^{d_1},\quad \mbox{for all}\quad m \geq 1.$$
In particular, putting $d_1=1$ and $d_2 = d(N)$,
\begin{equation}\label{eq:1vsdN}
\sum_{i\geq m} Q_i^{d(N)}\leq_{st}\sum_{i\geq m} Q_i^{1},\quad \mbox{for all}\quad m \geq 1.
\end{equation}
Let $\XX^N$ and $\YY^N$ denote random variables following the stationary distribution of two systems with $N$ servers under the JSQ$(d(N))$ and JSQ$(1)$ schemes, respectively. 
We will verify the tightness criterion stated in Lemma~\ref{lem:tightcond}.
Note that since $\XX^N$ takes value in $\mathcal{S}\subset [0,1]^\infty$, which is compact with respect to the product topology, Prohorov's theorem implies that $\big\{\XX^N\big\}_{N\geq 1}$ is tight with respect to the product topology.
To verify the condition in~\eqref{eq:smalltail}, note that the system under the JSQ$(1)$ scheme is essentially a collection of $N$ independent M/M/1 systems. Therefore, for each $k\geq 1$,
\begin{align*}
\varlimsup_{N\to\infty}\mathbbm{P}\Big(\sum_{i\geq k}X_i^N>\varepsilon\Big)
\leq \varlimsup_{N\to\infty}\mathbbm{P}\Big(\sum_{i\geq k}Y_i^N>\varepsilon\Big)
= (1-\lambda)\sum_{i\geq k}\lambda^i.
\end{align*}
Since $\lambda<1$, taking the limit $k\to\infty$, the right side of the above inequality tends to zero, and hence, the condition in~\eqref{eq:smalltail} is verified.

(ii) Now observe that since $\big\{\pi^{\sss d(N)}\big\}_{N\geq 1}$ is tight, any subsequence has a convergent further subsequence. 
Let $\big\{\pi^{\sss d(N_n)}\big\}_{n\geq 1}$ be any such convergent subsequence, with $\big\{N_n\big\}_{n\geq 1}\subseteq\N$, such that $\pi^{\sss d(N_n)}\dto\hat{\pi}$ as $n\to\infty$. We will show that $\hat{\pi}$ is unique and equals the measure $\pi^\star$, as defined in the statement of Proposition~\ref{th:interchange}.
Notice that if $\qq^{\sss d(N_n)}(0)\sim\pi^{\sss d(N_n)}$, then $\qq^{\sss d(N_n)}(t)\sim\pi^{\sss d(N_n)}$ for all $t\geq 0$.
Thus, $\hat{\pi}$ is an invariant distribution of the deterministic process $\big\{\qq(t)\big\}_{t\geq 0}$.
 This in conjunction with the global stability in Lemma~\ref{lem:global-stab} implies that $\hat{\pi}$ must be the fixed point of the fluid limit.  
Thus, we have shown the convergence of the stationary measure.
\end{proof}

\section{Diffusion-Limit Proofs}\label{sec:diffusion} 
In this section we prove the diffusion-limit results for the JSQ$(d(N))$ scheme stated in Theorem~\ref{diffusionjsqd}, and the almost necessity condition for diffusion-level optimality stated in Theorem~\ref{th:diff necessary}.
As noted in Subsection~\ref{subsec:strategy}, the diffusion limit 
for the ordinary JSQ policy is obtained in~\cite[Theorem~2]{EG15}, and characterized by~\eqref{eq:diffusionjsqd}.
Therefore it suffices to prove the universality property stated in the next proposition. 
\begin{proposition}\label{prop:samedif}
If $d(N)/(\sqrt{N}\log(N))\to\infty$ as $N\to\infty$, then the JSQ$(d(N))$ scheme and the ordinary JSQ policy have the same diffusion limit.
\end{proposition}
The proof of the above proposition follows similar lines as that of Proposition~\ref{prop:samefluid}, leveraging again the S-coupling results from Section~\ref{sec:coupling}, and
involves three steps:
\begin{enumerate}[{\normalfont (i)}]
\item First we show that if $n(N)/\sqrt{N}\to 0$ as $N\to\infty$, then the MJSQ$(n(N))$ scheme has the same diffusion limit as the ordinary JSQ policy.
\item Then we use Corollary~\ref{cor:bound} to prove that as long as $n(N)/\sqrt{N}\to 0$, \emph{any} scheme from the class CJSQ$(n(N))$ has the same diffusion limit as the ordinary JSQ policy.
\item Next we establish using Propositions~\ref{prop:stoch-ord2} and~\ref{prop:differ} that if $d(N)/(\sqrt{N}\log(N))\to\infty$ as $N\to\infty$, then for \emph{some} $n(N)$ with $n(N)/\sqrt{N}\to 0$, the  JSQ$(d(N))$ scheme and the JSQ$(n(N),d(N))$ scheme have the same diffusion limit. The proposition then follows by observing that the JSQ$(n(N),d(N))$ scheme belongs to the class CJSQ$(n(N))$.
\end{enumerate}

\begin{proof}[Proof of Proposition~\ref{prop:samedif}]
To show Claim~(i) above, define $\bar{N}=N-n(N)$ and $\bar{\lambda}(\bar{N})=\lambda(N)$.
As mentioned earlier, the MJSQ$(n(N))$ scheme with $N$ servers can be thought of as the ordinary JSQ policy with $\bar{N}$ servers and arrival rate $\bar{\lambda}(\bar{N})$.
Also, since $n(N)/\sqrt{N}\to 0$,
\begin{align*}
\frac{\bar{N}-\bar{\lambda}(\sqrt{\bar{N}})}{\bar{N}}=\frac{N-n(N)-\lambda(N)}{\sqrt{N-n(N)}}\to \beta>0\quad \text{as}\quad \bar{N}\to\infty.
\end{align*}
Furthermore, observe that the diffusion limit of the JSQ policy in \cite[Theorem~2]{EG15} as given in~\eqref{eq:diffusionjsqd} is characterized by the parameter $\beta>0$, and hence the diffusion limit of the MJSQ$(n(N))$ scheme is the same as that of the ordinary JSQ policy.

Observe from the diffusion limit of the JSQ policy that if $\beta>0$, then for any buffer capacity $b\geq 2$, and suitable initial state as described in~Theorem~\ref{diffusionjsqd}, the cumulative overflow is negligible, i.e., for any $t\geq 0$, $L^N(t)\pto 0$.
Indeed observe that if $b\geq 2$, and $\big\{\bQ_2^N(0)\big\}_{N\geq 1}$ is a tight sequence, then the sequence of processes $\big\{\bQ_2^N(t)\big\}_{t\geq 0}$ is stochastically bounded.
Therefore, on any finite time interval, there will be only $\Op(\sqrt{N})$ servers with queue length more than one, whereas, for an overflow event to occur all the $N$ servers must have at least two pending tasks.
Therefore, for any $t\geq 0$,
\begin{align*}
\limsup_{N\to\infty}\Pro{L^N(t)>0}&\leq \limsup_{N\to\infty}\Pro{\sup_{s\in[0,t]}Q^N_2(s)=N}\\
&\leq \limsup_{N\to\infty}\Pro{\sup_{s\in[0,t]}\bQ^N_2(s)=\sqrt{N}}=0.
\end{align*} 
Finally, since the above fact is implied by the diffusion limit only, the same holds for the MJSQ$(n(N))$ scheme.
Therefore, using the lower and upper bounds in Corollary~\ref{cor:bound} we arrive at Claim~(ii).

Finally, choose
$$n(N)=\frac{N\log N}{d(N)},$$
 and consider the JSQ$(n(N),d(N))$ scheme. 
Since $d(N)/( \sqrt{N}\log N)\to\infty$, it is clear that $n(N)/\sqrt{N}\to 0$ as $N\to\infty$.
Again, if $\Delta^N(T)$ denotes the cumulative number of times that the JSQ($d(N)$) scheme and JSQ$(n(N),d(N))$ scheme differ in decision up to time $T$, then Proposition~\ref{prop:differ} yields
\begin{equation}\label{eq:propdif}
\begin{split}
\Pro{\Delta^N(T)\geq \varepsilon \sqrt{N}\given A^N(T)}&\leq \frac{A^N(T)}{\varepsilon \sqrt{N}}\left(1-\frac{n(N)}{N}\right)^{ d(N)}\\
&\leq \frac{A^N(T)}{\varepsilon \sqrt{N}}\left(1-\frac{\log(N)}{d(N)}\right)^{ d(N)}\\
&\leq \frac{A^N(T)}{\varepsilon N} \sqrt{N}\left(1-\frac{\log(N)}{d(N)}\right)^{ d(N)}.
\end{split}
\end{equation}
Since $\big\{A^N(T)/N\big\}_{N\geq 1}$ is a tight sequence of random variables, and
\begin{align*}
&\sqrt{N}\left(1-\frac{\log(N)}{d(N)}\right)^{ d(N)}\to 0,
\quad\text{as}\quad N\to\infty,\\
\iff &\frac{1}{2}\log(N)+d(N)\log\left(1-\frac{\log(N)}{d(N)}\right)\to-\infty,\quad\text{as}\quad N\to\infty,\\
\Longleftarrow\hspace{.15cm}& \frac{1}{2}\log N - \frac{\log(N)}{d(N)}\times d(N)\to-\infty,\quad\text{as}\quad N\to\infty,
\end{align*}
from~\eqref{eq:propdif}, $\Delta^N(T)/N\pto 0$. 
Therefore, by invoking Proposition~\ref{prop:stoch-ord2}, we obtain Claim~(iii).
The proof is then completed by observing that the JSQ$(n(N),d(N))$ scheme belongs to the class CJSQ$(n(N))$.
\end{proof}

We next prove that the growth condition $d(N)/( \sqrt{N}\log N)\to\infty$ is nearly necessary: for any $d(N)$ such that $d(N)/(\sqrt{N}\log N)\to 0$ as $N\to\infty$, the diffusion limit of the JSQ$(d(N))$ scheme differs from that of the ordinary JSQ policy. 
Note that it is enough to consider the truncated system where any arrival to a server with at least two tasks is discarded, since the truncated system and the original system have the same diffusion limit~\cite{MBLW15}. 

Now consider the JSQ($d(N)$) scheme for some $d(N)$ such that $d(N)/(\sqrt{N}\log N)\to 0$ as $N\to\infty$, and assume on the contrary, the hypothesis that the process $\big\{(Q_1^{\sss d(N)}(t)-N)/\sqrt{N}, Q_2^{\sss d(N)}(t)/\sqrt{N}\big\}_{t\geq 0}$ converges to the diffusion limit corresponding that of the JSQ policy.
From a high level, the idea is to show that if the processes $(N-Q^{\sss d(N)}_1(\cdot))$ and $Q^{\sss d(N)}_2(\cdot)$ are $\Op(\sqrt{N})$, then in any finite time interval the number of tasks assigned to a server with queue length at least one, by the JSQ$(d(N))$ scheme with $d(N)/(\sqrt{N}\log N)\to 0$ does not scale with $\sqrt{N}$, which then immediately proves that the diffusion limit cannot coincide with that of the ordinary JSQ policy.

To formalize the above idea, we first define an artificial scheme below, which will serve as an asymptotic  lower bound to the number of servers with queue length two in a system following the JSQ$(d(N))$ scheme, under the hypothesis that the diffusion limit of the JSQ$(d(N))$ coincides with that of the ordinary JSQ policy.
For any nonnegative sequence $c(N)$, define a scheme $\Pi(c(N))$ which 
\begin{enumerate}[{\normalfont (i)}]
\item At each external arrival, assigns the task to a server having queue length one with probability $(1-c(N)/N)^{\sss d(N)}$, and else discards it (ties can be broken randomly),
\item If a departure occurs from a server with queue length one, then it immediately makes the server busy with a dummy arrival, i.e., essentially $\Pi(c(N))$ prohibits any server to remain idle.
\end{enumerate}
We use a coupling argument to show the following:
\begin{lemma}\label{lem:upper}
For any nonnegative sequence $c(N)$ with $c(N)/ \sqrt{N}\to \infty$ as $N\to\infty$, there exists a common probability space, such that for any $T>0$, 
$$\Pro{\sup_{t\in [0,T]}\left\{Q_2^{\sss d(N)}(t)-Q_2^{\sss \Pi(c(N))}(t)\right\}\geq 0}\longrightarrow 1\quad\mathrm{as}\quad N\to\infty,$$
provided $Q_2^{\sss d(N)}(0)\geq Q_2^{\sss \Pi(c(N))}(0)$ for all sufficiently large $N$,
and the hypothesis that the sequence of processes $\big\{(N-Q_1^{\sss d(N)}(t))/\sqrt{N}\big\}_{t\geq 0}$ and $\big\{Q_2^{\sss d(N)}(t)/\sqrt{N}\big\}_{t\geq 0}$ are stochastically bounded.
\end{lemma}
In order to prove Lemma~\ref{lem:upper}, we first S-couple the two systems under schemes $\Pi(c(N))$ and JSQ($d(N)$) respectively. 
Now at each external arrival, to assign the task in the two systems in a coupled way, draw a single uniform$[0,1]$ random variable $U$, independent of any other processes. 
\begin{itemize}
\item Under the JSQ($d(N)$) scheme, if $u<(Q_1^{\sss d(N)}/N)^{\sss d(N)}-(Q_2^{\sss d(N)}/N)^{\sss d(N)}$, then assign the task to a server with queue length one, if 
\begin{equation}\label{eq:idleserver}
(Q_1^{\sss d(N)}/N)^{\sss d(N)}-(Q_2^{\sss d(N)}/N)^{\sss d(N)}<U<1-(Q_2^{\sss d(N)}/N)^{\sss d(N)},
\end{equation}
then assign the task to an idle server, and otherwise discard it. 
This preserves the statistical law of the JSQ($d(N)$) scheme with a buffer size $b=2$.
Indeed note that according to the above rule the probability that an incoming task will be assigned
to some server with queue length zero, one, and two, are respectively given by $(1-(Q_1^{\sss d(N)}/N)^{\sss d(N)})$, $(Q_1^{\sss d(N)}/N)^{\sss d(N)}-(Q_2^{\sss d(N)}/N)^{\sss d(N)}$, and $(Q_2^{\sss d(N)}/N)^{\sss d(N)}$.
\item Under the scheme $\Pi(c(N))$, if $U<(1-c(N)/N)^{\sss d(N)}$, then assign the incoming task to a server with queue length one, otherwise discard it.
Clearly, the statistical law of the $\Pi(c(N))$ scheme is preserved by this rule.
\end{itemize}
\begin{proof}[Proof of Lemma~\ref{lem:upper}]
Fix any $T\geq 0$. Now the proof follows in two steps:

(i) First assume that at each external arrival up to time $T$, whenever an incoming task joins a server with queue length one, under the $\Pi(c(N))$ scheme, then so does the incoming task under JSQ($d(N)$) scheme. 
In that case, since the two systems are S-coupled, by forward induction on event times, it can be seen that $Q_2^{\sss d(N)}(t)\geq Q_2^{\sss \Pi(c(N))}(t)$ for all $0\leq t\leq T$, provided $Q_2^{\sss d(N)}(0)\geq Q_2^{\sss \Pi(c(N))}(0)$.

(ii) 
Now, for any $T\geq0$, 
according to the hypothesis, both $\sup_{t\in [0,T]}Q_2^{\sss d(N)}(t)$ and $\sup_{t\in[0,T]}\big\{N-Q_1^{\sss d(N)}(t)\big\}$  are $\Op(\sqrt{N})$. 
Also, since $c(N)/\sqrt{N}\to\infty$, it is straightforward to check that
\begin{equation}\label{eq:pi(c(N))}
\liminf_{N\to\infty}\Pro{\sup_{t\in [0,T]}\left(\frac{Q_1^{\sss d(N)}(t)}{N}\right)^{\sss d(N)}-\left(\frac{Q_2^{\sss d(N)}(t)}{N}\right)^{\sss d(N)}\geq \left(1-\frac{c(N)}{N}\right)^{\sss d(N)}}=1.
\end{equation}

Note that the probabilities that an incoming task joins a server with queue length one are given by 
$(Q_1^{\sss d(N)}(t)/N)^{\sss d(N)}-(Q_2^{\sss d(N)}(t)/N)^{\sss d(N)}$ and $(1-c(N)/N)^{\sss d(N)}$ for the JSQ($d(N)$) and the $\Pi(c(N))$ scheme respectively. 
Informally speaking, due to the above coupling,~\eqref{eq:pi(c(N))} then implies that with high probability, on any finite time interval, whenever an external incoming task joins a server with queue length one under the $\Pi(c(N))$ scheme, then so does the incoming task under the JSQ($d(N)$) scheme. 
 Therefore, from Part~(i) above, we can say
 \begin{align*}
&\Pro{\sup_{t\in [0,T]}\big\{Q_2^{\sss d(N)}(t)-Q_2^{\sss \Pi(c(N))}(t)\big\}\geq 0}\\
 &\geq \Pro{\sup_{t\in [0,T]}\left(\frac{Q_1^{\sss d(N)}(t)}{N}\right)^{\sss d(N)}-\left(\frac{Q_2^{\sss d(N)}(t)}{N}\right)^{\sss d(N)}\geq \left(1-\frac{c(N)}{N}\right)^{\sss d(N)}}\\
 &\longrightarrow 1 \quad\text{as}\quad N\to\infty.
 \end{align*}

This shows that if for any $T\geq 0$, if $\sup_{t\in [0,T]}Q_2^{\sss d(N)}(t)$ and $\sup_{t\in[0,T]}\left\{N-Q_1^{\sss d(N)}(t)\right\}$  are  $\Op(\sqrt{N})$, then with probability tending to one as $N\to\infty$, up to time $t$, the process $\big\{Q_2^{\sss \Pi(c(N))}(t)\big\}_{0\leq t\leq T}$ is indeed a lower bound for the process  $\big\{Q_2^{\sss d(N)}(t)\big\}_{0\leq t\leq T}$, and hence by our hypothesis, the proof is complete. 
\end{proof}

\begin{proof}[Proof of Theorem~\ref{th:diff necessary}]
Fix any positive sequence $d(N)$ such that $d(N)/(\sqrt{N}\log N)\to 0$ as $N\to\infty$.
Assume the hypothesis that for the JSQ($d(N)$) scheme, the process
$$\big\{(Q_1^{\sss d(N)}(t)-N)/\sqrt{N}, Q_2^{\sss d(N)}(t)/\sqrt{N}\big\}_{t\geq 0}$$ 
converges to the appropriate diffusion limit corresponding to that of the ordinary JSQ policy.
We will show that under this hypothesis, the process $\big\{Q_2^{\sss d(N)}(t)/\sqrt{N}\big\}_{t\geq 0}$ is not stochastically bounded, which will then lead to a contradiction. 

In order to show this, we will choose an appropriate $c(N)$ such that $c(N)/\sqrt{N}\to\infty$ as $N\to\infty$, and the process $\big\{Q_2^{\sss \Pi(c(N))}(t)/\sqrt{N}\big\}_{t\geq 0}$ is not stochastically bounded.
The conclusion then follows by the application of Lemma~\ref{lem:upper}.

Observe that the martingale decomposition of the scaled $Q_2^{\sss \Pi(c(N))}(\cdot)$ process can be written as
\begin{equation}\label{eq:mart-modif}
\bQ_2^{\sss \Pi(c(N))}(t)=\bQ_2^{\sss \Pi(c(N))}(0)+\frac{M^N(t)}{\sqrt{N}}+\frac{\lambda(N) t}{\sqrt{N}}\left(1-\frac{c(N)}{N}\right)^{\sss d(N)}-\int_0^t\bQ_2^{\sss \Pi(c(N))}(s)ds,
\end{equation}
where $\bQ^{\sss \Pi(c(N))}_2(t)=Q^{\sss \Pi(c(N))}_2(t)/\sqrt{N}.$ 
Now write $c(N)=g(N)\sqrt{N}$, for some $g(N)\to\infty$ (to be chosen later),  and $d(N)=\sqrt{N}\log(N)/\omega(N)$, where $\omega(N)=\sqrt{N}\log(N)/d(N)\to\infty$ as $N\to\infty$. Therefore, we write \eqref{eq:mart-modif} as
\begin{equation}
\begin{split}
\bQ_2^{\sss \Pi(c(N))}(t)=\bQ_2^{\sss \Pi(c(N))}(0)+\frac{M^N(t)}{\sqrt{N}}+\frac{\lambda(N) t}{\sqrt{N}}\left(1-\frac{g(N)}{\sqrt{N}}\right)^{\frac{\sqrt{N}\log N}{\omega(N)}}-\int_0^t\bQ_2^{\sss \Pi(c(N))}(s)ds.
\end{split}
\end{equation}
Observe that for any $t\geq 0$,
\begin{align*}
&\lim_{N\to\infty}\frac{\lambda(N) t}{\sqrt{N}}\left(1-\frac{g(N)}{\sqrt{N}}\right)^{\frac{\sqrt{N}\log N}{\omega(N)}}\\
&=t\lim_{N\to\infty}\exp\left[\log (\sqrt{N}-\beta)+\frac{\sqrt{N}\log N}{\omega(N)}\log\left(1-\frac{g(N)}{\sqrt{N}}\right)\right]\\
&=t\lim_{N\to\infty}\exp\left[\log (\sqrt{N}-\beta)-\frac{g(N)\log N}{\omega(N)}-o\left(\frac{g(N)\log N}{\omega(N)}\right)\right]
\end{align*}
Choosing $g(N)$ such that $g(N)/\omega(N)\to 0$ implies 
$$\frac{\lambda(N) t}{\sqrt{N}}\left(1-\frac{g(N)}{\sqrt{N}}\right)^{\frac{\sqrt{N}\log N}{\omega(N)}}\to\infty,\quad \text{as}\quad N\to\infty.$$
Note that for any $\omega(N)$, this choice of $g(N)$ is feasible (choose $g(N)=\sqrt{\omega(N)}$, say).
Furthermore, the process $\big\{M^N(t)/\sqrt{N}\big\}_{t\geq 0}$ in \eqref{eq:mart-modif} is stochastically bounded due to the martingale FCLT \cite[Theorem~7.1]{EK2009} and our hypothesis. 
Now we can conclude that for the above choices of $g(N)$ and $\omega(N)$, the process $\big\{\bQ_2^{\sss \Pi(c(N))}(t)\big\}_{t\geq 0}$, and hence the process $\big\{\bQ_2^{\sss d(N)}(t)\big\}_{t\geq 0}$ (due to Lemma~\ref{lem:upper}) is not stochastically bounded. 
Therefore, the limit does not coincide with the limit of the scaled $Q_2^{\sss \jsq}$-process.
\end{proof}

\section{Conclusion}\label{sec:conclusion}
In the present paper we have established universality properties
for power-of-$d$ load balancing schemes in many-server systems.
Specifically, we considered a system of $N$~parallel exponential
servers and a single dispatcher which assigns arriving tasks to the
server with the shortest queue among $d(N)$ randomly selected servers.
We developed a novel stochastic coupling construction to bound the
difference in the queue length processes between the JSQ policy
($d = N$) and a scheme with an arbitrary value of~$d$.
As it turns out, a direct comparison between the JSQ policy
and a JSQ($d$) scheme is a significant challenge.
Hence, we adopted a two-stage approach based on a novel class
of schemes which always assign the incoming task to one of the
servers with the $n(N) + 1$ smallest number of tasks.
Just like the JSQ($d(N)$) scheme, these schemes may be thought
of as `sloppy' versions of the JSQ policy.
Indeed, the JSQ($d(N)$) scheme is guaranteed to identify the
server with the minimum number of tasks, but only among
a randomly sampled subset of $d(N)$ servers.
In contrast, the schemes in the above class only guarantee that
one of the $n(N) + 1$ servers with the smallest number of tasks
is selected, but across the entire system of $N$ servers.
We showed that the system occupancy processes for an intermediate
blend of these schemes are simultaneously close on a $g(N)$ scale
($g(N) = N$ or $g(N) = \sqrt{N}$) to both the JSQ policy
and the JSQ($d(N)$) scheme for suitably chosen values of $d(N)$
and $n(N)$ as function of $g(N)$.
Based on the latter asymptotic universality, it then sufficed to
establish the fluid and diffusion limits for the ordinary JSQ policy.
Thus deriving the fluid limit of the ordinary JSQ policy, and using the above coupling argument we establish the fluid limit of the JSQ$(d(N))$ scheme in a regime with
$d(N) \to \infty$ as $N \to \infty$, along with the corresponding
fixed point.
The fluid limit turns out not to depend on the exact growth rate
of $d(N)$, and in particular coincides with that for the ordinary JSQ policy.
We further leveraged the coupling to prove that the diffusion limit
in the Halfin-Whitt regime with $d(N)/(\sqrt{N} \log (N)) \to \infty$
as $N \to \infty$ corresponds to that for the JSQ policy.
These results indicate that the optimality of the JSQ policy can be
preserved at the fluid-level and diffusion-level while reducing the
overhead by nearly a factor~O($N$) and O($\sqrt{N} / \log(N)$),
respectively.
In future work we plan to extend the results to heterogeneous
servers and non-exponential service requirement distributions.
We also intend to pursue extensions to network scenarios
and server-task compatibility constraints.

\section*{Acknowledgment}
This research was financially supported by an ERC Starting Grant and by The Netherlands Organization for Scientific Research (NWO) through TOP-GO grant 613.001.012 and Gravitation Networks grant 024.002.003. Dr.~Whiting was supported in part by an Australian Research grant DP-1592400 and in part by a Macquarie University Vice-Chancellor Innovation Fellowship. 

{
\bibliographystyle{apa}

\bibliography{bibl}

\begin{thebibliography}{}

\bibitem[\protect\astroncite{Banerjee and Mukherjee}{2018a}]{BM18b}
Banerjee, S. and Mukherjee, D. (2018a).
\newblock {Join-the-shortest queue diffusion limit in Halfin-Whitt regime:
  Sensitivity on the heavy traffic parameter}.
\newblock {\em arXiv:1809.01739}.

\bibitem[\protect\astroncite{Banerjee and Mukherjee}{2018b}]{BM18}
Banerjee, S. and Mukherjee, D. (2018b).
\newblock {Join-the-shortest queue diffusion limit in Halfin-Whitt regime: Tail
  asymptotics and scaling of extrema}.
\newblock {\em Ann. Appl. Probab. (to appear)}.

\bibitem[\protect\astroncite{Bramson et~al.}{2012}]{BLP12}
Bramson, M., Lu, Y., and Prabhakar, B. (2012).
\newblock {Asymptotic independence of queues under randomized load balancing}.
\newblock {\em Queueing Systems}, 71(3):247--292.

\bibitem[\protect\astroncite{Braverman}{2018}]{Braverman18}
Braverman, A. (2018).
\newblock {Steady-state analysis of the join the shortest queue model in the
  Halfin-Whitt regime}.
\newblock {\em arXiv:1801.05121}.

\bibitem[\protect\astroncite{Brightwell and Luczak}{2012}]{BL12}
Brightwell, G. and Luczak, M. (2012).
\newblock {The supermarket model with arrival rate tending to one}.
\newblock {\em arXiv:1201.5523}.

\bibitem[\protect\astroncite{Ephremides et~al.}{1980}]{EVW80}
Ephremides, A., Varaiya, P., and Walrand, J. (1980).
\newblock {A simple dynamic routing problem}.
\newblock {\em IEEE Transactions on Automatic Control}, 25(4):690--693.

\bibitem[\protect\astroncite{Eschenfeldt and Gamarnik}{2016}]{EG16}
Eschenfeldt, P. and Gamarnik, D. (2016).
\newblock {Supermarket queueing system in the heavy traffic regime. Short queue
  dynamics}.
\newblock {\em arXiv: 1610.03522}.

\bibitem[\protect\astroncite{Eschenfeldt and Gamarnik}{2018}]{EG15}
Eschenfeldt, P. and Gamarnik, D. (2018).
\newblock {Join the shortest queue with many servers. The heavy
  traffic-asymptotics}.
\newblock {\em Math. Oper. Res}, 43(3):867--886.

\bibitem[\protect\astroncite{Ethier and Kurtz}{2009}]{EK2009}
Ethier, S.~N. and Kurtz, T.~G. (2009).
\newblock {\em {Markov Processes: Characterization and Convergence}}.
\newblock John Wiley {\&} Sons.

\bibitem[\protect\astroncite{Gamarnik et~al.}{2016}]{GTZ16}
Gamarnik, D., Tsitsiklis, J., and Zubeldia, M. (2016).
\newblock {Delay, memory and messaging tradeoffs in distributed service
  systems}.
\newblock In {\em Proc. ACM SIGMETRICS}, pages 1--12.

\bibitem[\protect\astroncite{Graham}{2005}]{G05}
Graham, C. (2005).
\newblock {Functional central limit theorems for a large network in which
  customers join the shortest of several queues}.
\newblock {\em Probability Theory and Related Fields}, 131(1):97--120.

\bibitem[\protect\astroncite{Gupta et~al.}{2007}]{GHSW07}
Gupta, V., Harchol-Balter, M., Sigman, K., and Whitt, W. (2007).
\newblock {Analysis of join-the-shortest-queue routing for web server farms}.
\newblock {\em Performance Evaluation}, 64(9-12):1062--1081.

\bibitem[\protect\astroncite{Halfin and Whitt}{1981}]{HW81}
Halfin, S. and Whitt, W. (1981).
\newblock {Heavy-traffic limits for queues with many exponential servers}.
\newblock {\em Operations Research}, 29(3):567--588.

\bibitem[\protect\astroncite{Hoeffding}{1963}]{H63}
Hoeffding, W. (1963).
\newblock {Probability inequalities for sums of bounded random variables}.
\newblock {\em Journal of the American Statistical Association},
  58(301):13--30.

\bibitem[\protect\astroncite{Hunt and Kurtz}{1994}]{HK94}
Hunt, P. and Kurtz, T. (1994).
\newblock {Large loss networks}.
\newblock {\em Stochastic Processes and their Applications}, 53(2):363--378.

\bibitem[\protect\astroncite{Jonckheere}{2006}]{Jonckheere06}
Jonckheere, M. (2006).
\newblock {Insensitive versus efficient dynamic load balancing in networks
  without blocking}.
\newblock {\em Queueing Systems}, 54(3):193--202.

\bibitem[\protect\astroncite{Liptser and Shiryaev}{1989}]{LS89}
Liptser, R. and Shiryaev, A. (1989).
\newblock {\em {Theory of Martingales}}.
\newblock Springer.

\bibitem[\protect\astroncite{Luczak and McDiarmid}{2006}]{LM06}
Luczak, M.~J. and McDiarmid, C. (2006).
\newblock {On the maximum queue length in the supermarket model}.
\newblock {\em The Annals of Probability}, 34(2):493--527.

\bibitem[\protect\astroncite{Luczak and Norris}{2005}]{LN05}
Luczak, M.~J. and Norris, J. (2005).
\newblock {Strong approximation for the supermarket model}.
\newblock {\em The Annals of Applied Probability}, 15(3):2038--2061.

\bibitem[\protect\astroncite{Luh and Pippenger}{2014}]{LP14}
Luh, K. and Pippenger, N. (2014).
\newblock {Large-deviation bounds for sampling without replacement}.
\newblock {\em The American Mathematical Monthly}, 121(5):449--454.

\bibitem[\protect\astroncite{Mitzenmacher}{1996}]{Mitzenmacher1996}
Mitzenmacher, M. (1996).
\newblock {\em {The power of two choices in randomized load balancing}}.
\newblock PhD thesis, University of California, Berkeley.

\bibitem[\protect\astroncite{Mitzenmacher}{2001}]{Mitzenmacher01}
Mitzenmacher, M. (2001).
\newblock {The power of two choices in randomized load balancing}.
\newblock {\em IEEE Transactions on Parallel and Distributed Systems},
  12(10):1094--1104.

\bibitem[\protect\astroncite{Mukherjee et~al.}{2016a}]{MBLW16-4}
Mukherjee, D., Borst, S.~C., van Leeuwaarden, J. S.~H., and Whiting, P.~A.
  (2016a).
\newblock {Asymptotic optimality of power-of-d load balancing in large-scale
  systems}.
\newblock {\em arXiv:1612.00722}.

\bibitem[\protect\astroncite{Mukherjee et~al.}{2016b}]{MBLW15}
Mukherjee, D., Borst, S.~C., van Leeuwaarden, J. S.~H., and Whiting, P.~A.
  (2016b).
\newblock {Universality of load balancing schemes on the diffusion scale}.
\newblock {\em J. Appl. Probab.}, 53(4).

\bibitem[\protect\astroncite{Pang et~al.}{2007}]{PTRW07}
Pang, G., Talreja, R., and Whitt, W. (2007).
\newblock {Martingale proofs of many-server heavy-traffic limits for Markovian
  queues}.
\newblock {\em Prob. Surveys}, 4:193--267.

\bibitem[\protect\astroncite{Sparaggis et~al.}{1994}]{towsley}
Sparaggis, P.~D., Towsley, D., and Cassandras, C.~G. (1994).
\newblock {Sample path criteria for weak majorization}.
\newblock {\em Advances in Applied Probability}, 26(1):155--171.

\bibitem[\protect\astroncite{Towsley}{1995}]{Towsley95}
Towsley, D. (1995).
\newblock {Application of majorization to control problems in queueing
  systems}.
\newblock In Chr{\'{e}}tienne, P., Coffman, E.~G., Lenstra, J.~K., and Liu, Z.,
  editors, {\em Scheduling Theory and its Applications}, chapter~14. John Wiley
  {\&} Sons, Chichester.

\bibitem[\protect\astroncite{Towsley et~al.}{1992}]{Towsley1992}
Towsley, D., Sparaggis, P., and Cassandras, C. (1992).
\newblock {Optimal routing and buffer allocation for a class of finite capacity
  queueing systems}.
\newblock {\em IEEE Transactions on Automatic Control}, 37(9):1446--1451.

\bibitem[\protect\astroncite{van Leeuwaarden and Knessl}{2011}]{LK11}
van Leeuwaarden, J. S.~H. and Knessl, C. (2011).
\newblock {Transient behavior of the Halfin–Whitt diffusion}.
\newblock {\em Stochastic Processes and their Applications}, 121(7):1524--1545.

\bibitem[\protect\astroncite{van Leeuwaarden and Knessl}{2012}]{LK12}
van Leeuwaarden, J. S.~H. and Knessl, C. (2012).
\newblock {Spectral gap of the Erlang A model in the Halfin-Whitt regime}.
\newblock {\em Stochastic Systems}, 2(1):149--207.

\bibitem[\protect\astroncite{Vvedenskaya et~al.}{1996}]{VDK96}
Vvedenskaya, N.~D., Dobrushin, R.~L., and Karpelevich, F.~I. (1996).
\newblock {Queueing system with selection of the shortest of two queues: An
  asymptotic approach}.
\newblock {\em Problemy Peredachi Informatsii}, 32(1):20--34.

\bibitem[\protect\astroncite{Weber}{1978}]{W78}
Weber, R.~R. (1978).
\newblock {On the optimal assignment of customers to parallel servers}.
\newblock {\em Journal of Applied Probability}, 15(2):406--413.

\bibitem[\protect\astroncite{Whitt}{1986}]{Whitt86}
Whitt, W. (1986).
\newblock {Deciding which queue to join: Some counterexamples}.
\newblock {\em Operations Research}, 34(1):55--62.

\bibitem[\protect\astroncite{Whitt}{2002}]{W02}
Whitt, W. (2002).
\newblock {\em {Stochastic-Process Limits}}.
\newblock Springer Series in Operations Research and Financial Engineering.
  Springer-Verlag, New York.

\bibitem[\protect\astroncite{Winston}{1977}]{Winston77}
Winston, W. (1977).
\newblock {Optimality of the shortest line discipline}.
\newblock {\em Journal of Applied Probability}, 14(1):181--189.

\bibitem[\protect\astroncite{Ying et~al.}{2015}]{YSK15}
Ying, L., Srikant, R., and Kang, X. (2015).
\newblock {The power of slightly more than one sample in randomized load
  balancing}.
\newblock In {\em Proc. IEEE INFOCOM}, pages 1131--1139.

\end{thebibliography}
}
\end{document}